\documentclass[a4paper,reqno]{amsart}
\usepackage{amsmath,amsthm,amssymb,amsfonts}
\usepackage[english]{babel}
\usepackage{enumitem}
\setlist[enumerate]{leftmargin=*}

\usepackage{hyperref}

\title[The Hausdorff--Young inequality]{The Hausdorff--Young inequality on Lie groups}

\author[M. G. Cowling]{Michael G. Cowling}
\address[M. G. Cowling]{School of Mathematics and Statistics \\ University of New South Wales \\ Sydney NSW 2052 \\ Australia}
\email{m.cowling@unsw.edu.au }

\author[A. Martini]{Alessio Martini}
\address[A. Martini]{School of Mathematics \\ University of Birmingham \\ Edgbaston \\ Birmingham B15 2TT \\ United Kingdom}
\email{a.martini@bham.ac.uk }

\author[D. M\"uller]{Detlef M\"uller}
\address[D. M\"uller]{Mathematisches Seminar \\ Christian-Albrechts-Universit\"at zu Kiel \\ Ludewyg--Meyn-Str.\ 4 \\ D-24118 Kiel \\ Germany}
\email{mueller@math.uni-kiel.de }

\author[J. Parcet]{Javier Parcet}
\address[J. Parcet]{Instituto de Ciencias Matem\'aticas \\ Consejo Superior de Investigaciones Cient\'ificas \\ C/ Nicol\'as Cabrera 13-15 \\ Madrid \\ Spain}
\email{javier.parcet@icmat.es}


\thanks{The first-named author was supported by the Australian Research Council (grant DP170103025). The second-named author was supported by the Engineering and Physical Sciences Research Council (grant EP/P002447/1). The fourth-named author was supported by the Europa Excelencia Grant MTM2016-81700-ERC and the CSIC Grant PIE-201650E030.}

\newtheorem{theorem}{Theorem}[section]
\newtheorem{corollary}[theorem]{Corollary}
\newtheorem{lemma}[theorem]{Lemma}
\newtheorem{proposition}[theorem]{Proposition}
\theoremstyle{definition}
\newtheorem{remarkx}[theorem]{Remark}
\newenvironment{remark}
  {\pushQED{\qed}\remarkx}
  {\popQED\endremarkx}

\numberwithin{equation}{section}

\newcommand{\R}{\mathbb{R}}
\newcommand{\C}{\mathbb{C}}
\newcommand{\Z}{\mathbb{Z}}
\newcommand{\N}{\mathbb{N}}

\newcommand{\T}{\mathbb{T}}  %torus

\newcommand{\Heis}{\mathbb{H}}  %Heisenberg group

\newcommand{\Four}{\mathcal{F}}    %Fourier
\newcommand{\Hilb}{\mathcal{H}}    % Hilbert space
\newcommand{\Lin}{\mathcal{L}}   % bounded linear operators

\newcommand{\tc}{\,:\,}

\newcommand{\pg}{pg} % polynomial growth
\newcommand{\unit}{\mathrm{u}}   % unitary (dual, Fourier transform)
\newcommand{\loc}{\mathrm{loc}}

\newcommand\group[1]{\mathrm{#1}}
\newcommand\Lie[1]{\mathfrak{#1}}

\newcommand{\Sch}{\mathcal{S}}      %Schatten
\newcommand{\HS}{\mathrm{HS}}       % Hilbert-Schmidt

\newcommand{\Cv}{\mathrm{Cv}}                % convolutors
\newcommand{\VN}{\mathrm{VN}}                % von Neumann

\newcommand{\chrfn}{\mathbf{1}}        % characteristic function

\DeclareMathOperator{\Inn}{Inn} % inner automorphisms
\newcommand{\Ad}{\operatorname{Ad}}     % adjoint representation (group)
\DeclareMathOperator{\supp}{supp}    % support
\DeclareMathOperator{\trace}{trace}  % trace

\DeclareMathOperator*{\esssup}{ess\,sup}

\DeclareFontFamily{U}{mathb}{\hyphenchar\font45}
\DeclareFontShape{U}{mathb}{m}{n}{
      <5> <6> <7> <8> <9> <10> gen * mathb
      <10.95> mathb10 <12> <14.4> <17.28> <20.74> <24.88> mathb12
      }{}
\DeclareSymbolFont{mathb}{U}{mathb}{m}{n}
\DeclareMathSymbol{\bigast}        {1}{mathb}{"06}

\begin{document}

\begin{abstract}
We prove several results about the best constants in the Hausdorff--Young inequality for noncommutative groups.
In particular, we establish a sharp local central version for compact Lie groups, and extend known results for the Heisenberg group.
In addition, we prove a universal lower bound to the best constant for general Lie groups.
\end{abstract}
\maketitle

\section{Introduction}

For $f \in L^1(\R^n)$, define the Fourier transform $\hat f$ of $f$ by
\[
\hat f(\xi) = \int_{\R^n} f(x) \, e^{2\pi i \xi \cdot x} \, dx \qquad\forall \xi \in \R^n.
\]
Then the Riemann--Lebesgue lemma states that $\hat f \in C_0(\R^n)$ and
\[
\| \hat f \|_\infty \leq \| f\|_1.
\]
Further, the Plancherel theorem entails that if $f \in L^2(\R^n)$, then
\[
\| \hat f \|_2  =  \| f \|_2.
\]
Suppose that $1 \leq p \leq 2$ and $p'$ is the conjugate exponent to $p$, that is, $1/p' = 1 - 1/p$.
Then interpolation implies the Hausdorff--Young inequality, namely,
\begin{equation}\label{eq:HY}
\| \hat f \|_{p'} \leq C \|f\|_{p}
\end{equation}
for all $f \in L^p(\R^n)$, where $C \leq 1$. We denote the best constant for this inequality, that is, the smallest possible value of $C$, by $H_p(\R^n)$.
This was found many years after the original result.
We define the Babenko--Beckner constant $B_p$ by
\[
B_p =   \frac{ p^{1/2p} }{ (p')^{ 1/2p' } }.
\]
Then $B_p < 1$ when $1 < p < 2$.

\begin{theorem}[Babenko \cite{Bab-61}, Beckner \cite{Bec-1975}]
For all $p \in [1,2]$,
\[
H_p(\R^n) = (B_p)^n.
\]
\end{theorem}

Babenko treated the case where $p' \in 2 \Z$, and Beckner proved the general case.
The extremal functions are gaussians; see \cite{lieb_1990} for an alternative proof.

One can extend the Babenko--Beckner theorem to more general contexts than $\R^n$, such as locally compact abelian groups $G$.
For instance, the best constant $H_p(G)$ for the inequality \eqref{eq:HY} when $G = \R^a \times \T^b \times \Z^c$ is $(B_p)^a$.
The extremal functions are of the form $\gamma \otimes \chi \otimes \delta$, where $\gamma$ is a gaussian on $\R^a$,  $\chi$ is a character of $\T^b$, and $\delta$ is the characteristic function of a point in $\Z^c$.

For nonabelian groups, matters are more complicated, in part because the interpretation of the $L^{q}$ norm of the Fourier transform for $q \in (2,\infty)$ is trickier. We refer the reader to Section \ref{s:FLq} below for details. General versions of the Hausdorff--Young inequality \eqref{eq:HY} were obtained by Kunze \cite{Kun-1958} and Terp \cite{terp_1980} for arbitrary locally compact groups $G$, and a number of works in the literature are devoted to the study of the corresponding best constants $H_p(G)$. It is known, at least in the unimodular case, that $H_p(G)<1$ for $p \in (1,2)$ if and only if $G$ has no compact open subgroups \cite{Russo-1974,fournier_1977}. On the other hand, when $H_p(G)$ is not $1$, its value is known only in few cases, and typically only for exponents $p$ whose conjugate exponent is an even integer; in addition, as shown by Klein and Russo, extremal functions need not exist \cite{KR-1978}.

Recently, various authors considered local versions of the Hausdorff--Young inequality. Namely, for each neighbourhood $U$ of the identity $e \in G$, define $H_p(G;U)$ as the best constant in the inequality \eqref{eq:HY} with the additional support constraint $\supp f \subseteq U$, and let $H_p^\loc(G)$ be the infimum of the constants $H_p(G;U)$. Clearly $H_p^\loc(G) \leq H_p(G)$, and equality holds whenever $G$ has a contractive automorphism. For other groups, however, the inequality may be strict, which makes the study of $H_p^\loc(G)$ interesting also for groups where $H_p(G) = 1$, such as compact groups. Indeed, in the case of the torus $G = \T^n$, the value of $H_p^\loc(G)$ is known and is strictly less than $1$ for $p \in (1,2)$.

\begin{theorem}[Andersson \cite{And-PhD,And-1994}, Sj\"olin \cite{Sjo-1995}, Kamaly \cite{Kam-2000}]\label{thm:local-HY-Tn}
For all $p \in [1,2]$,
\[
H_p^\loc(\T^n) = (B_p)^n.
\]
\end{theorem}

Here we are interested in analogues of the above result for noncommutative Lie groups $G$.
We also study what happens when 
additional symmetries are imposed
by restricting
to functions $f$ on $G$ which are invariant under a compact group $K$ of automorphisms of $G$. Let us denote by $H_{p,K}(G)$ and $H_{p,K}^\loc(G)$ the corresponding global and local best Hausdorff--Young constants.
Note that the original constants $H_p(G)$ and $H_{p}^\loc(G)$ correspond to the case where $K$ is trivial. When $K$ is nontrivial, \emph{a priori} the new constants $H_{p,K}(G)$ and $H_{p,K}^\loc(G)$ might be smaller. However we can prove a universal lower bound, which is independent of the symmetry group $K$ and depends only on $p$ and the dimension of $G$.

\begin{theorem}\label{thm:local-HY}
Let $G$ be a Lie group and $K$ be a compact group of automorphisms of $G$. For all $p \in [1,2]$,
\[
H_{p,K}^\loc(G) \geq (B_p)^{\dim(G)}.
\]
\end{theorem}

Recall that a function $f$ on a group $G$ is \emph{central} if $f(xy) = f(yx)$, that is, if $f$ is invariant under the group $\Inn(G)$ of inner automorphisms of $G$.
Garc{\'\i}a-Cuerva, Marco and Parcet \cite{GCMP-2003} and Garc{\'\i}a-Cuerva and Parcet \cite{GCP-2004} studied the Hausdorff--Young inequality for compact semi\-simple Lie groups $G$ restricted to central functions; in particular, they obtained the inequality $H^\loc_{p,\Inn(G)}(G) > 0$, which they applied to answer questions about Fourier type and cotype of operator spaces (see also \cite{Par-2006}). Theorem \ref{thm:local-HY} gives a substantially more precise lower bound to $H^\loc_{p,\Inn(G)}(G)$. As a matter of fact, in this case we can prove that equality holds.

\begin{theorem}\label{thm:local-central-HY-compact-Lie}
Suppose that $G$ is a compact connected Lie group.
Then, for all $p \in [1,2]$,
\[
H_{p,\Inn(G)}^\loc(G) = (B_p)^{\dim(G)}.
\]
\end{theorem}

Note on the one hand that, in the abelian case $G = \T^n$, all functions are central, so Theorem \ref{thm:local-central-HY-compact-Lie} extends Theorem \ref{thm:local-HY-Tn}. On the other hand, it would be interesting to know whether the result holds also without the restriction to central functions.

More generally, one may ask whether the inequality in Theorem \ref{thm:local-HY} is actually an equality for an arbitrary Lie group $G$. As a matter of fact, the equality
\[
H_{p,K}^\loc(G) = (B_p)^{\dim(G)}
\]
holds for arbitrary $G$ and $K$ whenever $p' \in 2\Z$, as a consequence of a recent result of Bennett, Bez, Buschenhenke, Cowling and Flock \cite{BBBCF_2018}
and the relation between the best constants for the Young and the Hausdorff--Young inequalities (see Proposition \ref{prop:basic} below). In particular, by interpolation,
\[
H_{p,K}^\loc(G) < 1
\]
for all $p \in (1,2)$ and arbitrary $G$ and $K$ with $\dim(G) > 0$. 
Moreover, the equality
\begin{equation}\label{eq:hy_conjecture}
H_{p}(G) = H_p^\loc(G) = (B_p)^{\dim(G)}
\end{equation}
holds when $p' \in 2\Z$ for all Lie groups $G$ with a contractive automorphism (which are nilpotent---see \cite{Siebert-1986}), and also for all solvable Lie groups $G$ admitting a chain of closed subgroups
\[
\{e\} = G_0 < G_1 < \dots < G_{n-1} < G_n = G,
\]
where $G_j$ is normal in $G_{j+1}$ and $G_{j+1} / G_j$ is isomorphic to $\R$ (here $n=\dim(G)$). For many of those groups $G$, the upper bound $H_{p}(G) \leq (B_p)^{\dim(G)}$ for $p'\in 2\Z$ was proved in \cite{KR-1978}, but the question of the lower bound was left open there, except for the Heisenberg groups. Hence Theorem \ref{thm:local-HY} proves the sharpness of a number of results in \cite{KR-1978}.

The Heisenberg groups $\Heis_n$ are among the simplest examples of groups in the above class. Nevertheless, determining the value of $H_{p}(\Heis_n) = H_{p}^\loc(\Heis_n)$ appears to be a nontrivial problem when $p' \notin 2\Z$, and is related to a similar problem for the so-called Weyl transform. Recall that the Weyl transform $\rho$ on $\C^n$ maps functions on $\C^n$ to integral operators on $L^2(\R^n)$ \cite{Folland-1989},
and an inequality of Hausdorff--Young type can be proved for $\rho$  \cite{KR-1978,Russo-1979}:
for all $p \in [1,2]$,
\begin{equation}\label{eq:HY_weyl}
\| \rho(f) \|_{\Sch^{p'}(L^2(\R^n))} \leq C \|f\|_{L^p(\C^{n})},
\end{equation}
where $\Sch^{q}(\Hilb)$ denotes the $q$th Schatten class of operators on the Hilbert space $\Hilb$, and $C \leq 1$. As above, we can define $W_p(\C^n)$ as the best constant in \eqref{eq:HY_weyl}, as well as corresponding local and symmetric versions $W_{p}^\loc(\C^n), W_{p,K}(\C^n), W_{p,K}^\loc(\C^n)$. 
A scaling argument (see Proposition \ref{prp:weylheisenberg} below) then shows that, for all compact subgroups $K$ of the unitary group $\group{U}(n)$,
\begin{equation}\label{eq:HY_heis_weyl}
H_{p,K}(\Heis_n) = B_p \, W_{p,K}(\C^n)
\end{equation}
(here $\group{U}(n)$ acts naturally on $\C^n$ and the first layer of $\Heis_n$). So the problem of determining the best Hausdorff--Young constants for the Heisenberg group $\Heis_n$ is equivalent to the analogous problem for the Weyl transform.
In particular, \eqref{eq:HY_heis_weyl} and Theorem \ref{thm:local-HY} yield that
\[
W_{p,K}(\C^n) \geq (B_p)^{2n}
\]
for all $p \in [1,2]$. As an indication that equality may well hold, here we prove the following local result.

\begin{theorem}\label{thm:local-HY-weyl}
Let $K$ be a compact subgroup of $\group{U}(n)$. Then, for all $p \in [1,2]$,
\[
W_{p,K}^\loc(\C^n) \geq (B_p)^{2n}.
\]
Moreover, if $K \supseteq \group{U}(1) \times \dots \times \group{U}(1)$, then, for all $p \in [1,2]$,
\[
W_{p,K}^\loc(\C^n) = (B_p)^{2n}.
\]
\end{theorem}

Functions on $\C^n$ or $\Heis_n$ that are invariant under $\group{U}(1) \times \dots \times \group{U}(1)$ are called polyradial. Equality in Theorem \ref{thm:local-HY-weyl} is obtained as a consequence of the following weighted Hausdorff--Young inequality for polyradial functions $f$:
\begin{equation}\label{eq:weighted_HY_weyl}
\| \rho(f) \|_{\Sch^{p'}(\R^n)} \leq (B_p)^{2n} \| f e^{(\pi/2)|\cdot|^2} \|_{L^p(\C^n)}.
\end{equation}
Unfortunately we have not found a way to remove the weight and obtain the equality $W_{p,K}(\C^n) = W_{p,K}^\loc(\C^n)$ for arbitrary $p \in [1,2]$; note however that $W_{p,K}(\C^n) = W_{p,K}^\loc(\C^n) = (B_p)^{2n}$ when $p' \in 2\Z$, as proved in \cite{KR-1978}.

Both cases where we can prove equalities in Theorems \ref{thm:local-central-HY-compact-Lie} and \ref{thm:local-HY-weyl} for general $p \in [1,2]$ correspond to Gelfand pairs (see, for example, \cite{carcano_1987}): indeed, central functions on a compact group $G$ and polyradial functions on the Heisenberg group $\Heis_n$ form commutative subalgebras of the respective convolution algebras $L^1(G)$ and $L^1(\Heis_n)$. It seems a reasonable intermediate question to ask for best constants in Hausdorff--Young inequalities in the context of Gelfand pairs, since here the group Fourier transform reduces to the Gelfand transform for the corresponding commutative algebra of invariant functions, which makes the $L^q$ norm of the Fourier transform in these settings more accessible. Indeed, in both the proofs of Theorems \ref{thm:local-central-HY-compact-Lie} and \ref{thm:local-HY-weyl}, this additional commutativity allows one to relate the group Fourier transform and the Weyl transform with the Euclidean Fourier transform, for which the Babenko--Beckner result is available. Regrettably, even in the case of polyradial functions on the Heisenberg group we are not able yet to fully answer the question. Indeed, as we discuss in Section \ref{s:weyl}, in this case it seems unlikely that the best Hausdorff--Young constant on the Heisenberg group can be obtained by a direct reduction to the corresponding sharp Euclidean estimate, and new ideas appear to be needed.

As for the universal lower bound of Theorem \ref{thm:local-HY}, the intuitive idea behind its proof is that, at smaller and smaller scales, the group structure of a Lie group $G$ looks more and more like the abelian group structure of its Lie algebra $\Lie{g}$, whence $H_p^\loc(G)$ is likely to be related to $H_p(\Lie{g}) = (B_p)^{\dim(G)}$. Indeed, a scaling argument based on this idea readily yields the analogue of Theorem \ref{thm:local-HY} for Young's convolution inequality (see the discussion in Section \ref{s:FLq} below).
This appears to have been overlooked in \cite{KR-1978}, where a number of upper bounds for Young constants on Lie groups are proved, which are actually equalities in view of this observation.

The additional complication with the Hausdorff--Young inequality is that it involves the $L^q$ norm of the Fourier transform. While it is reasonably clear that,  at small scales, the noncommutative convolution on $G$ approximates the commutative convolution on $\Lie{g}$, the same is not so evident for the Fourier transform: indeed, if the group Fourier transform is defined, as it is common, in terms of irreducible unitary representations, then it is not immediately clear how to relate the representation theories of $G$ and $\Lie{g}$ for an arbitrary Lie group $G$, let alone the corresponding Fourier transforms and $L^q$ norms thereof. Here we completely bypass the problem, by characterising the $L^q$ norm of the Fourier transform in terms of an operator norm of a fractional power of an integral operator, acting on functions on $G$:
\begin{equation}\label{eq:FTnorm}
\| \hat f \|_q^q = \| |L_f \Delta^{1/q}|^q \|_{1 \to \infty}.
\end{equation}
Here $L_f$ is the operator of convolution on the left by $f$ and $\Delta$ is the operator of multiplication by the modular function of $G$. A transplantation argument, not dissimilar from those in \cite{mitjagin_1974,KST-1982,martini_joint_2017}, allows us to relate the operator $L_f \Delta^{1/q}$ on $G$ to its counterpart on $\Lie{g}$ and obtain the desired lower bound.

Although it might be evident to some experts in noncommutative integration, we are not aware of the characterisation \eqref{eq:FTnorm} being explicitly observed before. What is interesting about \eqref{eq:FTnorm} is that it allows one to access the $L^q$ norm of the Fourier transform through properties of a more ``geometric'' convolution-multiplication operator on $G$, which appears to be more tractable. As a matter of fact, when dealing with convolution, one can use induction-on-scales methods to completely determine the best local constants for the Young convolution inequality on any Lie group $G$; this remarkable result has been recently proved in \cite{BBBCF_2018}, as a corollary of a more general result for nonlinear Brascamp--Lieb inequalities. It would be interesting to know whether similar methods could be applied to the Hausdorff--Young inequality on noncommutative Lie groups as well.

\subsection*{Plan of the paper}
In Section \ref{s:FLq} we discuss the definition of the $L^q$ norm of the Fourier transform for an arbitrary Lie group, by comparing a number of definitions available in the literature, and prove the characterisation \eqref{eq:FTnorm}; we also present a proof of the universal lower bound of Theorem \ref{thm:local-HY}, as well as its analogue for the Young convolution inequality, and discuss relations between best constants for Young and Hausdorff--Young inequalities. The sharp local central Hausdorff--Young inequality for arbitrary compact Lie groups (Theorem \ref{thm:local-central-HY-compact-Lie}) is proved in Section \ref{s:compact}; to better explain the underlying idea without delving into technicalities, the proof of the abelian case (Theorem \ref{thm:local-HY-Tn}) is briefly revisited in Section \ref{s:torus}. Finally, in Section \ref{s:weyl} we discuss the relations between Hausdorff--Young constants for the Heisenberg group and the Weyl transform and prove Theorem \ref{thm:local-HY-weyl}, together with the weighted inequality \eqref{eq:weighted_HY_weyl} for polyradial functions.

\section{\texorpdfstring{$L^q$}{Lq} norm of the Fourier transform}\label{s:FLq}

Let $G$ be a Lie group (or, more generally, a separable locally compact group) with a fixed left Haar measure. 
In order to discuss best Hausdorff--Young constants in this generality, we first need to clarify what is meant by the ``Fourier transform'' in this setting and how Hausdorff--Young inequalities --- even the endpoint ones, such as the Plancherel formula --- can be stated in this context.

A common way to generalise the Fourier transformation to this setting exploits irreducible unitary representations of $G$ (see, for example, \cite{lipsman_1974} or \cite[Chapter 7]{folland_course_1995} for a survey). Namely, let $\widehat G_\unit$ be the ``unitary dual'' of $G$, that is, the set of (equivalence classes of) irreducible unitary representations of $G$, endowed with the Fell topology and the Mackey Borel structure. The (unitary) Fourier transform $\Four_\unit f$ of a function $f \in L^1(G)$ is then defined as the operator-valued function on $\widehat G_\unit$ given by
\[
\widehat G_\unit \ni \pi \mapsto \pi(f) = \int_G f(x) \pi(x) \,dx \in \Lin(\Hilb_\pi);
\]
here $\Lin(\Hilb_\pi)$ denotes the space of bounded linear operators on the Hilbert space $\Hilb_\pi$ on which the representation $\pi$ acts, and integration is with respect to the Haar measure. In case $G$ is unimodular and type I (this includes the cases where $G$ is abelian or compact), the Plancherel formula can be stated in the form
\begin{equation}\label{eq:plancherel_rep}
\|f\|^2_{L^2(G)} = \int_{\widehat G_\unit} \| \pi(f) \|_{\HS(\Hilb_\pi)}^2 \,d\pi
\end{equation}
for all $f \in L^1 \cap L^2(G)$. Here $\HS(\Hilb_\pi)$ denotes the space of Hilbert--Schmidt operators on $\Hilb_\pi$, and integration on $\widehat G_\unit$ is with respect to a suitable measure, called the Plancherel measure, which is uniquely determined by the above formula; in addition, the Fourier transformation $f \mapsto \Four_\unit f$ extends to an isometric isomorphism between $L^2(G)$ and the direct integral $L^2_\unit(\widehat G) := \int^\oplus_{\widehat G_\unit} \HS(\Hilb_\pi) \,d\pi$. Interpolation then leads to the Hausdorff--Young inequality
\begin{equation}\label{eq:Lq_rep}
\|\Four_\unit f\|_{L^{p'}_\unit(\widehat G)} := \left(\int_{\widehat G_\unit} \| \pi(f) \|_{\Sch^{p'}(\Hilb_\pi)}^{p'} \,d\pi\right)^{1/p'} \leq C \|f\|_{L^p(G)}
\end{equation}
when $1 < p < 2$, where $C = 1$; here, for all $q \in [1,\infty]$, $\Sch^q(\Hilb_\pi)$ denotes the $q$th Schatten class of operators on $\Hilb_\pi$, and the operator-valued $L^q$-spaces $L^q_\unit(\widehat G)$ are defined in terms of measurable fields of operators as in \cite{lipsman_1974}. The fact that the spaces $L^q_\unit(\widehat G)$ constitute a complex interpolation family, that is,
\begin{equation}\label{eq:interpol_unitLp}
[L^{q_0}_\unit(\widehat G),L^{q_1}_\unit(\widehat G)]_\theta = L^q_\unit(\widehat G)
\end{equation}
with equal norms for $q_0,q_1,q \in [1,\infty]$, $\theta \in (0,1)$, $1/q = (1-\theta)/q_0 + \theta/q_1$, readily follows from standard interpolation results for vector-valued Lebesgue spaces and Schatten classes (see, for example, \cite{triebel_1978,hytonen_2016,Pisier-Xu-2003}) and the structure of the measurable field of separable Hilbert spaces $\pi \mapsto \Hilb_\pi$ \cite[Proposition 7.19]{folland_course_1995}.

In the case where $G$ is not unimodular, under suitable type I assumptions it is possible to prove a Plancherel formula similar to \eqref{eq:plancherel_rep}, where the right-hand side is adjusted by means of ``formal dimension operators'' \cite{tatsuuma_1972,kleppner_lipsman_1972,kleppner_lipsman_1973,duflo_moore_1976,Fuehr_2005}. Analogous modifications of \eqref{eq:Lq_rep} lead to a version of the Hausdorff--Young inequality that has been studied in a number of works \cite{Eymard-Terp-1979,Russo-1979,Inoue-1992,Fuehr_2006,Baklouti-Ludwig-Scuto-Smaoui-2007}. 

When $G$ is not type I, 
 the above approach to the Plancherel formula based on irreducible unitary representation theory does not work as neatly.
This however does not prevent one from studying the Hausdorff--Young inequality. Indeed, what is possibly the first appearance in the literature of the Hausdorff--Young inequality in a noncommutative setting, that is, the work of Kunze \cite{Kun-1958} for arbitrary unimodular locally compact groups (not necessarily of type I), does not express the Fourier transform in terms of irreducible unitary representations, but uses instead the theory of noncommutative integration (the same theory was used in earlier works of Mautner \cite{mautner_unitary_1950} and Segal \cite{segal-1950} to express the Plancherel formula). This point of view was subsequently developed by Terp \cite{terp_1980} to cover the case of non-unimodular groups and more recently has been further extended to the context of locally compact quantum groups \cite{caspers_2013,cooney_2010}.

One way of thinking of noncommutative $L^q$ spaces is as complex interpolation spaces between a von Neumann algebra $M$ and its predual $M_*$ (which play the role of $L^\infty$ and $L^1$ respectively)  \cite{terp_1982,Kosaki-1984,Izumi-1997,Pisier-Xu-2003}. In general this requires establishing a ``compatibility'' between $M$ and $M_*$, which may involve a number of choices, but in our case there appears to be a natural way to proceed (see also \cite{forrest_2011,daws_2011}). Namely, the von Neumann algebra $\VN(G)$ of $G$ (that is, the weak${}^*$-closed $*$-subalgebra of $\Lin(L^2(G))$ of the operators which commute with right translations) can be identified with the space $\Cv^2(G)$ of left convolutors of $L^2(G)$, that is, those distributions on $G$ which are left convolution kernels of $L^2(G)$-bounded operators. Moreover, the predual $\VN(G)_*$ can be identified with the Fourier algebra $A(G)$, an algebra of continuous functions on $G$ defined by Eymard \cite{Eymard-1964} for arbitrary locally compact groups $G$. Now $A(G)$ and $\Cv^2(G)$ are naturally compatible as spaces of distributions on $G$ (see \cite[Propositions (3.26) and (3.27)]{Eymard-1964}), so we can use complex interpolation to define Fourier--Lebesgue spaces of distributions on $G$: for $q\in[1,\infty]$, we set
\[
\Four L^q(G) = \begin{cases}
A(G) &\text{if } q=1,\\
\Cv^2(G) &\text{if } q=\infty,\\
[A(G),\Cv^2(G)]_{1-1/q} &\text{if } 1 < q < \infty.
\end{cases}
\]
One can check that this definition corresponds to Izumi's left $L^p$ spaces \cite{Izumi-1997,Izumi-1998} for the von Neumann algebra $\VN(G)$ with respect to the Plancherel weight, and therefore it matches the construction given in \cite{caspers_2013,cooney_2010} for quantum groups. In particular $\Four L^2(G) = L^2(G)$ with equality of norms (see \cite[Section 5]{Izumi-1998} and \cite[Proposition 2.21(iii)]{caspers_2013}; this corresponds to the Plancherel theorem), while clearly $L^1(G) \subseteq \Cv^2(G)$ with norm-decreasing embedding. Interpolation then leads to the following formulation of the Hausdorff--Young inequality: $L^p(G) \subseteq \Four L^{p'}(G)$ and
\begin{equation}\label{eq:HY_FLq}
\| f \|_{\Four L^{p'}(G)} \leq C \| f\|_{L^p(G)}
\end{equation}
where $C= 1$ and $p \in [1,2]$.

We then define the $L^p$ Hausdorff--Young constant $H_p(G)$ on the group $G$ as the minimal constant $C$ for which \eqref{eq:HY_FLq} holds for all $f \in L^p(G)$. Similarly, if $U$ is a neighbourhood of the identity in $G$, we let $H_p(G;U)$ be the minimal constant $C$ in \eqref{eq:HY_FLq} when $f$ is constrained to have support in $U$, and define the local $L^p$ Hausdorff--Young constant $H_p^\loc(G)$ as the infimum of the constants $H_p(G;U)$ where $U$ ranges over the neighbourhoods of the identity of $G$.

The approach to Hausdorff--Young constants via $\Four L^q$ spaces is consistent with the unitary Fourier transformation approach described above, when the latter is applicable. Indeed, as discussed in \cite[Theorems 2.1 and 3.1]{lipsman_1974}, in the case where $G$ is unimodular and type I, the unitary Fourier transformation $\Four_\unit$ induces isometric isomorphisms $\Cv^2(G) \cong L^\infty_\unit(\widehat G)$ and $A(G) \cong L^1_\unit(\widehat G)$, besides the Plancherel isomorphism $L^2(G) \cong L^2_\unit(\widehat G)$ (analogous results in the nonunimodular case can be found in \cite[Theorems 3.48 and 4.12]{Fuehr_2005}); so by interpolation $\Four_u$ induces an isometric isomorphism between $\Four L^q(G)$ and $L^q_\unit(\widehat G)$ for all $q \in [1,\infty]$. Hence defining Hausdorff--Young constants in terms of the inequality \eqref{eq:Lq_rep} would lead to the same constants $H_p(G)$ and $H_p^\loc(G)$ as those we have defined in terms of $\Four L^q$ spaces. On the other hand, the approach via $\Four L^q$ spaces does not require type I assumptions, or even separability, and can be applied to every locally compact group $G$.

There is an alternative characterisation of the noncommutative $L^q$ spaces associated to $\VN(G)$, namely as certain spaces $L^q_{\VN}(\widehat G)$ of (closed, possibly unbounded) operators on $L^2(G)$. This characterisation, which is that originally used in the works of Kunze and Terp on the Hausdorff--Young inequality, corresponds to Hilsum's approach to noncommutative $L^q$ spaces \cite{Hilsum-1981} based on Connes's ``spatial derivative'' construction \cite{Connes-1980} (the work of Kunze is actually based on an earlier version of the theory \cite{dixmier_1953,segal-1953} that only applies to semifinite von Neumann algebras). We will not enter into the details of this construction and only recall two important properties. First, if the operator $T$ belongs to $L^q_{\VN}(\widehat G)$ for some $q \in [1,\infty)$, then $|T|^q = (T^* T)^{q/2}$ belongs to $L^1_{\VN}(\widehat G)$ and
\begin{equation}\label{eq:LqL1}
\|T\|_{L^q_{\VN}(\widehat G)}^q = \||T|^q\|_{L^1_{\VN}(\widehat G)}.
\end{equation}
Moreover, for all $q \in [1,\infty]$, an isometric isomorphism from $\Four L^q(G)$ to $L^q_{\VN}(\widehat G)$ is given by
\begin{equation}\label{eq:LqFT}
f \mapsto L_f \Delta^{1/q},
\end{equation}
where $L_f$ is the left-convolution operator by $f$, and we identify the modular function $\Delta$ of $G$ with the corresponding multiplication operator (see \cite[Proposition 2.21(ii)]{caspers_2013}). Recall that convolution on $G$ is given by
\[
L_f \phi(x) = f * \phi(x) = \int_G f(xy) \, \phi(y^{-1}) \,dy,
\]
at least when $f$ and $\phi$ are in $C_c(G)$.

Note that, when $q=p'$, \eqref{eq:LqFT} matches the definitions by Kunze and by Terp of the $L^p$ Fourier transformation $\Four_p : L^p(G) \to L^{p'}_{\VN}(\widehat G)$ for $p \in [1,2]$ \cite{Kun-1958,terp_1980}. In other words, the $L^p$ Fourier transformation $\Four_p : L^p(G) \to L^{p'}_{\VN}(\widehat G)$ factorises as the inclusion map $L^p(G) \to \Four L^{p'}(G)$ and the isometric isomorphism $\Four L^{p'}(G) \to L^{p'}_{\VN}(\widehat G)$, whence the compatibility  with the Kunze--Terp approach of the above definition of the best Hausdorff--Young constants based on \eqref{eq:HY_FLq}.

Another consequence of the above discussion is the following characterisation of the $\Four L^q(G)$ norm in terms of a more ``concrete'' operator norm.

\begin{proposition}\label{prp:fouriernorm}
For all $q \in [1,\infty)$ and $f \in \Four L^q(G)$,
\begin{equation}\label{eq:nu_fouriernorm}
\| f \|_{\Four L^q(G)} =  \| |L_f \Delta^{1/q}|^q \|_{L^1(G) \to L^\infty(G)}^{1/q}.
\end{equation}
\end{proposition}
\begin{proof}
By \eqref{eq:LqL1} and \eqref{eq:LqFT},
\[
\| f \|_{\Four L^q(G)} = \| L_f \Delta^{1/q} \|_{L^q_{\VN}(\widehat G)} = \| |L_f \Delta^{1/q}|^q \|_{L^1_{\VN}(\widehat G)}^{1/q} = \| g \|_{A(G)}^{1/q},
\]
where $g \in A(G)$ satisfies $L_g \Delta = |L_f \Delta^{1/q}|^q$. On the other hand, the operator $L_g \Delta$ is given by
\[
L_g \Delta \phi(x) = \int_G g(xy) \, \Delta(y^{-1}) \, \phi(y^{-1}) \,dy = \int_G g(xy^{-1}) \, \phi(y) \,dy;
\]
since $L_g \Delta = |L_f \Delta^{1/q}|^q$ is a positive operator, the kernel $g$ must be a function of positive type (see, for example, \cite[Section 3.3]{folland_course_1995}), whence
\[
\|g\|_{A(G)} = g(e) = \|g\|_\infty = \| L_g \Delta \|_{L^1(G) \to L^\infty(G)}
\]
and we are done.
\end{proof}

A classical way of accessing Hausdorff--Young constants is through their relations with best constants in the Young convolution inequalities. Recall that, for a possibly nonunimodular group $G$, the $k$-linear version of Young's inequality takes the following form:
for all $p_1,\dots,p_k,r \in [1,\infty]$ such that $\sum_{j=1}^k 1/p_j' = 1/r'$,
\begin{equation}\label{eq:nu_young}
\Bigl\| \bigast_{j=1}^k (f_j \Delta^{\sum_{l=1}^{j-1} 1/p_l'}) \Bigr\|_{L^r(G)} \leq C \prod_{j=1}^k \|f_j\|_{L^{p_j}(G)}
\end{equation}
where $C \leq 1$ (see \cite[Lemma 1.1]{terp_1980}, or \cite[Corollary 2.3]{KR-1978} where the inequality is written for the right Haar measure). As in the case of the Hausdorff--Young inequality, we can define the Young constant $Y_{p_1,\dots,p_k}(G)$ for $G$ as the smallest constant $C$ for which \eqref{eq:nu_young} holds for all $f_1 \in L^{p_1}(G), \dots, f_k \in L^{p_k}(G)$, as well as the localised versions $Y_{p_1,\dots,p_k}(G;U)$ for neighbourhoods $U$ of the identity of $G$ (corresponding to the constraint $\supp f_1, \dots, \supp f_k \subseteq U$) and $Y^\loc_{p_1,\dots,p_k}(G)$.

Note that the above Young inequality \eqref{eq:nu_young} is ``dual'' to the following H\"older-type inequality for $\Four L^p$-spaces: for all $p_1,\dots,p_k,r \in [1,\infty]$ such that $\sum_{j=1}^k 1/p_j = 1/r$,
\begin{equation}\label{eq:nu_nchoelder}
\Bigl\| \bigast_{j=1}^k (f_j \Delta^{\sum_{l=1}^{j-1} 1/p_l}) \Bigr\|_{\Four L^r(G)} \leq \prod_{j=1}^k \|f_j\|_{\Four L^{p_j}(G)};
\end{equation}
this is a rephrasing of H\"older's inequality for Hilsum's noncommutative $L^p$ spaces,
\[
\| T_1 \cdots T_k \|_{L^r_{\VN}(\widehat G)} \leq \prod_{j=1}^k \| T_j \|_{L^{p_j}_{\VN}(\widehat G)}
\]
\cite[Proposition 8]{Hilsum-1981}, via the isomorphism \eqref{eq:LqFT} from $\Four L^q(G)$ to $L^q_{\VN}(\widehat G)$ and the identities
\begin{equation}\label{eq:conv_modular}
\Delta^{\alpha} (f*g) = (\Delta^\alpha f) * (\Delta^\alpha g) \qquad\text{and}\qquad L_{\Delta^{\alpha} f} = \Delta^{\alpha} L_f \Delta^{-\alpha},
\end{equation}
valid for all $\alpha \in \C$. Let us also recall that
\begin{equation}\label{eq:adj_involution}
L_{f^*} = L_{f}^*,
\end{equation}
where $f \mapsto f^*$ is the isometric conjugate-linear involution of $L^1(G)$ given by
\[
f^*(x) = \Delta^{-1}(x) \, \overline{f(x^{-1})}.
\]

The proposition below summarises a number of relations between Young and Hausdorff--Young constants that can be found in the literature, at least in particular cases (see, for example, \cite{Bec-1975} and \cite{KR-1978}), as well as corresponding local versions.

\begin{proposition}\label{prop:basic}
Let $G$ be a locally compact group.
\begin{enumerate}[label=(\roman*)]
\item\label{en:yhineq} For all $p_1,\dots,p_k,q \in [1,2]$ such that $\sum_j 1/p_j' = 1/q$,
\begin{align*}
Y_{p_1,\dots,p_k}(G) &\leq H_{q}(G) \, H_{p_1}(G) \cdots H_{p_k}(G), \\
Y_{p_1,\dots,p_k}^\loc(G) &\leq H_{q}^\loc(G) \, H_{p_1}^\loc(G) \cdots H_{p_k}^\loc(G).
\end{align*}
\item\label{en:yheq} For all $p \in [1,2)$ such that $p'=2k$, $k \in \Z$, if $p_1=\dots=p_k=p$, then
\begin{align*}
H_{p}(G) &= Y_{p_1,\dots,p_k}(G)^{1/k}, \\
H_{p}^\loc(G) &= Y_{p_1,\dots,p_k}^\loc(G)^{1/k}.
\end{align*}
\item\label{en:ext} If $N$ is a closed normal subgroup of $G$, then, for all $p_1,\dots,p_k \in [1,\infty]$ such that $\sum_{j=1}^k 1/p_j' \in [0,1]$,
\begin{align*}
Y_{p_1,\dots,p_k}(G) &\leq Y_{p_1,\dots,p_k}(N) \, Y_{p_1,\dots,p_k}(G/N),\\
Y_{p_1,\dots,p_k}^\loc(G) &\leq Y_{p_1,\dots,p_k}^\loc(N) \, Y_{p_1,\dots,p_k}^\loc(G/N),
\end{align*}
with equality when $G \cong N \times (G/N)$.
\end{enumerate}
\end{proposition}
\begin{proof}
\ref{en:yhineq}.
For all $f_1,\dots,f_k,g \in C_c(G)$, by \eqref{eq:HY_FLq} and \eqref{eq:nu_nchoelder},
\[\begin{split}
\left\langle \bigast_{j=1}^k (f_j \Delta^{\sum_{l=1}^{j-1} 1/p_l'}) , g \right\rangle  &\leq \left\| \bigast_{j=1}^k (f_j \Delta^{\sum_{l=1}^{j-1} 1/p_l'})  \right\|_{\Four L^q} \|  g \|_{\Four L^{q'}} \\
&\leq \| f_1\|_{\Four L^{p_1'}} \cdots \|f_k\|_{\Four L^{p_k'}} \|g\|_{\Four L^{q'}} \\
&\leq H_q(G) H_{p_1}(G) \cdots H_{p_k}(G) \|f_1\|_{L^{p_1}} \cdots \|f_k\|_{L^{p_k}} \|g\|_{L^q},
\end{split}\]
which proves that
\[
\left\|\bigast_{j=1}^k (f_j \Delta^{\sum_{l=1}^{j-1} 1/p_l'}) \right\|_{L^{q'}} \leq H_q(G) H_{p_1}(G) \cdots H_{p_k}(G) \|f_1\|_{L^{p_1}} \cdots \|f_k\|_{L^{p_k}},
\]
that is, $Y_{p_1,\dots,p_k}(G) \leq H_{q}(G) \, H_{p_1}(G) \cdots H_{p_k}(G)$. Note now that, if $f_1,\dots,f_k$ are supported in a neighbourhood $U$ of the identity, then $\bigast_{j=1}^k (f_j \Delta^{\sum_{l=1}^{j-1} 1/p_l'})$ is supported in $U^k$ and, to estimate its $L^{q'}$ norm, it is enough to test it against functions $g$ that are also supported in $U^k$; the same argument as above then also gives 
\[
Y_{p_1,\dots,p_k}(G;U) \leq H_{q}(G;U^k) \, H_{p_1}(G;U) \cdots H_{p_k}(G;U)
\]
and $Y_{p_1,\dots,p_k}^\loc(G) \leq H_{q}^\loc(G) \, H_{p_1}^\loc(G) \cdots H_{p_k}^\loc(G)$.

\ref{en:yheq}.
Part \ref{en:yhineq} gives us the inequality $H_{p}(G) \geq Y_{p_1,\dots,p_k}(G)^{1/k}$ and its local version.
On the other hand, for all $f \in C_c(G)$, if we define $\tilde f = \Delta^{1/p'} f^*$, then
\[
\|\tilde f\|_p = \|f\|_p
\]
and, by \eqref{eq:conv_modular} and \eqref{eq:adj_involution},
\[
L_{\tilde f} \Delta^{1/p'} = (L_f \Delta^{1/p'})^*.
\]
For all $j=1,\dots,k$, let $f_j$ be either $\tilde f$ or $f$, according to whether $k-j$ is odd or even, and define $g = \bigast_{j=1}^k (f_j \Delta^{(j-1)/p'})$. Then, since $p'=2k$,
\[\begin{split}
|L_f \Delta^{1/p'}|^{p'} &= [(L_f \Delta^{1/p'})^* (L_f \Delta^{1/p'}) ]^k \\
&= (L_{\tilde f} \Delta^{1/p'}) (L_{f} \Delta^{1/p'}) \cdots (L_{\tilde f} \Delta^{1/p'}) (L_{f} \Delta^{1/p'}) \\
&= |(L_{f_1} \Delta^{1/p'}) \cdots (L_{f_k} \Delta^{1/p'})|^2
\end{split}\]
and, by \eqref{eq:conv_modular},
\[
(L_{f_1} \Delta^{1/p'}) \cdots (L_{f_k} \Delta^{1/p'}) = L_g \Delta^{1/2}.
\]
So $|L_f \Delta^{1/p'}|^{p'} = |L_g \Delta^{1/2}|^2$ and, by \eqref{eq:nu_fouriernorm} and \eqref{eq:nu_young},
\[
\| f \|_{\Four L^{p'}}^{p'} = \| g \|_{\Four L^2}^2 = \| g \|_{L^2}^2 \leq Y_{p_1,\dots,p_k}(G)^2 \|f_1\|_{L^p}^2 \cdots \| f_k \|_{L^p}^2 = Y_{p_1,\dots,p_k}(G)^2 \|f\|_{L^p}^{p'},
\]
which gives the inequality $H_{p}(G) \leq Y_{p_1,\dots,p_k}(G)^{1/k}$. The same argument also gives $H_{p}(G;U) \leq Y_{p_1,\dots,p_k}(G;U)^{1/k}$ and $H_{p}^\loc(G) \leq Y_{p_1,\dots,p_k}^\loc(G)^{1/k}$.

\ref{en:ext}.
The inequalities are proved by a simple extension of Klein and Russo's argument for the case of semidirect products \cite[proof of Lemma 2.4]{KR-1978}, using the ``measure disintegration'' in \cite[Theorem (2.49)]{folland_course_1995}.
In the case of direct products, equalities follow by testing on tensor product functions (see \cite[Lemma 5]{Bec-1975}).
\end{proof}

The next lemma contains the fundamental approximation results that allow us to relate 
Hausdorff--Young constants on a Lie group $G$ and on its Lie algebra $\Lie{g}$ by means of a ``transplantation'' or ``blow-up'' technique.
The Lie algebra $\Lie{g}$ will be considered as an abelian group with addition, and the Lebesgue measure on $\Lie{g}$ is normalised so that the Jacobian determinant of the exponential map $\exp : \Lie{g} \to G$ is equal to $1$ at the origin. The context will make clear whether the notation for convolution, involution and convolution operators ($f*g$, $f^*$, $L_f$) refers to the group structure of $G$ or the abelian group structure of $\Lie{g}$.

Denote by $C_{\pg}([0,\infty))$ the space of continuous functions $\Phi : [0,\infty) \to \C$ with at most polynomial growth, that is, $|\Phi(u)| \leq C(1+u)^N$ for some $C,N \in (0,\infty)$ and all $u \in [0,\infty)$.

\begin{lemma}\label{lem:nu_localisation}
Let $G$ be a Lie group with Lie algebra $\Lie{g}$ of dimension $n$, and let $\exp : \Lie{g} \to G$ be the exponential map. Let $\Omega$ be an open neighbourhood of the origin in $\Lie{g}$ such that $\Omega = -\Omega$ and  $\exp|_\Omega : \Omega \to \exp(\Omega)$ is a diffeomorphism. For all $f \in C_c(\Lie{g})$, $\lambda \in (0,\infty)$, $\alpha \in \R$ and $p \in [1,\infty]$, define $f^{\lambda,p,\alpha} : G \to \C$ by
\begin{equation}\label{eq:nu_associatedp}
f^{\lambda,p,\alpha}(x) = \begin{cases}
\lambda^{-n/p} \Delta(x)^{-\alpha} f(\lambda^{-1} \exp|_\Omega^{-1}(x)) &\text{if } x \in \exp(\Omega),\\
0 &\text{otherwise.}
\end{cases}
\end{equation}
Set also $f^{\lambda,p} = f^{\lambda,p,0}$. Then the following hold.
\begin{enumerate}[label=(\roman*)]
\item\label{en:nu_asspnorm} For all $f \in C_c(\Lie{g})$, $\alpha \in \R$ and $p \in [1,\infty]$,
\begin{equation}\label{eq:nu_asspnormest}
\|f^{\lambda,p,\alpha}\|_{L^p(G)} \leq C_{\alpha,p,\Omega} \, \|f\|_{L^p(\Lie{g})}
\end{equation}
for all $\lambda \in (0,\infty)$, and
\begin{equation}\label{eq:nu_asspnormlim}
\|f^{\lambda,p,\alpha}\|_{L^p(G)} \to \|f\|_{L^p(\Lie{g})}
\end{equation}
as $\lambda \to 0$.

\item\label{en:nu_assinnerprod} For all $k \in \N$, $\alpha_1,\dots,\alpha_k,\beta\in \R$, $f_1,\dots,f_k,g \in C_c(\Lie{g})$,
\begin{equation}\label{eq:nu_innerconv}
\langle f_1^{\lambda,1,\alpha_1} * \cdots * f_k^{\lambda,1,\alpha_k} , g^{\lambda,\infty,\beta} \rangle_{L^2(G)} \to \langle f_1 * \cdots * f_k , g \rangle_{L^2(\Lie{g})} 
\end{equation}
as $\lambda \to 0$.

\item\label{en:nu_assfunccalc} For all $\alpha \in \R$, $f,g,h \in C_c(\Lie{g})$, $\Phi \in C_{\pg}([0,\infty))$,
\begin{equation}\label{eq:nu_continuousfc}
\langle \Phi(\Delta^\alpha L_{(f^{\lambda,1})^* * f^{\lambda,1}} \Delta^\alpha) g^{\lambda,2}, h^{\lambda,2}\rangle_{L^2(G)} \to  \langle\Phi(L_{f^* * f}) g, h\rangle_{L^2(\Lie{g})}
\end{equation}
as $\lambda \to 0$.

\item\label{en:nu_asspower} For all $\alpha \in \R$, $f,g,h \in C_c(\Lie{g})$ and $q \in [0,\infty)$,
\[
\lambda^{-n(q-1)} \langle  |L_{f^{\lambda,\infty}} \Delta^\alpha|^{q} g^{\lambda,1}, h^{\lambda,1}\rangle_{L^2(G)} \to  \langle|L_{f}|^{q} g, h\rangle_{L^2(\Lie{g})}
\]
as $\lambda \to 0$.
\end{enumerate}
\end{lemma}
\begin{proof}
Let $J : \Lie{g} \to \R$ denote the modulus of the Jacobian determinant of $\exp$, and define $\Delta_e : \Lie{g} \to (0,\infty)$ to be $\Delta \circ \exp$

\ref{en:nu_asspnorm}.
Note that
\[
\|f^{\lambda,p,\alpha}\|_{p}^{p} = \lambda^{-n} \int_\Omega |f(\lambda^{-1} X)|^{p} (J\Delta_e^{-\alpha p})(X) \,dX = \int_{\lambda^{-1}\Omega} |f(X)|^{p}  (J \Delta_e^{-\alpha p})(\lambda X) \,dX.
\]
From this, \eqref{eq:nu_asspnormest} follows (with $C_{\alpha,p,\Omega}^p = \sup_{\Omega} J \Delta_e^{-\alpha p}$), and \eqref{eq:nu_asspnormlim} follows as well because $f$ is compactly supported and $\lim_{X \to 0} (J \Delta_e^{-\alpha p})(X) = J(0) \Delta(e)^{-\alpha p} = 1$.

\ref{en:nu_assinnerprod}.
By the Baker--Campbell--Hausdorff formula,
\[
\exp(X_1) \cdots \exp(X_k) = \exp(X_1 + \dots + X_k + B(X_1,\dots,X_k)),
\]
where $B(X_1,\dots,X_k) = \sum_{m \geq 2} B_m(X_1,\dots,X_k)$ and, for all $m \geq 2$, $B_m(X_1,\dots,X_k)$ is a homogeneous polynomial function of $X_1,\dots,X_k$ of degree $m$; indeed we can find a sufficiently small neighbourhood $\tilde\Omega \subseteq \Omega$ of the origin in $\Lie{g}$ so that, if $X_1,\dots,X_k \in \tilde\Omega$, then $X_1 + \dots + X_k + B(X_1,\dots,X_k) \in \Omega$.

Note that
\begin{multline*}
\langle f_1^{\lambda,1,\alpha_1} * \cdots * f_k^{\lambda,1,\alpha_k} , g^{\lambda,\infty,\beta} \rangle_{L^2(G)} \\
= \int_{G^k} f_1^{\lambda,1,\alpha_1}(x_1) \, \cdots \, f_k^{\lambda,1,\alpha_k}(x_k) \, \overline{g^{\lambda,\infty,\beta}(x_1\cdots x_k)} \,dx_1 \dots \,dx_k
\end{multline*}
If $\lambda$ is sufficiently small that $\bigcup_{j=1}^k \lambda \supp f_j \subseteq \exp(\tilde\Omega)$, then the last integral may be rewritten as
\[
 \int_{\Lie{g}^k} \bar g\Biggl(\sum_{j=1}^k X_j + \lambda^{-1}B(\lambda X_1,\dots,\lambda X_k)\Biggr) \prod_{j=1}^k ( f_j(X_j) (J \Delta_e^{-\alpha_j-\beta})(\lambda X_j) ) \,dX_1\cdots \,dX_k.
\]
Since $\lambda^{-1} B(\lambda X_1,\dots,\lambda X_k) = \lambda \sum_{m \geq 2} \lambda^{m-2} B_m(X_1,\dots,X_k)$ tends to $0$ as $\lambda \to 0$, the last integral tends to
$\langle f_1 * \cdots * f_k , g \rangle_{L^2(\Lie{g})}$.

\ref{en:nu_assfunccalc}.
Note first that $\Delta^\alpha L_{(f^{\lambda,1})^* * f^{\lambda,1}} \Delta^\alpha$ is a nonnegative self-adjoint operator on $L^2(G)$ (which may be unbounded when $G$ is nonunimodular) and that, for all $N \in \N$, the $L^2$-domain of $(\Delta^\alpha L_{(f^{\lambda,1})^* * f^{\lambda,1}} \Delta^\alpha)^N$ contains all compactly supported functions in $L^2(G)$, so the left-hand side of \eqref{eq:nu_continuousfc} is well-defined. Note moreover that
\begin{equation}\label{eq:nu_star}
(f^{\lambda,p,\alpha})^* = (f^*)^{\lambda,p,1-\alpha}
\end{equation}
whence, by \eqref{eq:conv_modular},
\[\begin{split}
&\langle (\Delta^\alpha L_{(f^{\lambda,1})^* * f^{\lambda,1}} \Delta^\alpha)^N g^{\lambda,2}, h^{\lambda,2} \rangle_{L^2(G)} \\ 
&= \left\langle \left(\bigast_{j=1}^N ((f^*)^{\lambda,1,1-(2j-1)\alpha} * f^{\lambda,1,-(2j-1)\alpha}) \right) * g^{\lambda,1,-2N\alpha},  h^{\lambda,\infty} \right\rangle.
\end{split}\]
So, in the case where $\Phi(u) = u^N$ for some $N \in \N$, \eqref{eq:nu_continuousfc} follows from \eqref{eq:nu_innerconv}.

Note that, by shrinking $\Omega$ if necessary, we may assume that $\Omega$ and $\exp(\Omega)$ have compact closures in $\Lie{g}$ and $G$, and moreover the topological boundary of $\exp(\Omega)$ has null Haar measure (indeed shrinking $\Omega$ does not change the left-hand side of \eqref{eq:nu_continuousfc} for $\lambda$ sufficiently small). As in \cite[proof of Theorem 5.2]{martini_joint_2017}, we can now extend the diffeomorphism $\phi := \exp|_\Omega^{-1} : \exp(\Omega) \to \Omega$ to a diffeomorphism $\phi_* : U \to V$, where $U$ and $V$ are open sets in $G$ and $\Lie{g}$ containing $\exp(\Omega)$ and $\Omega$, and moreover $G \setminus U$ has null Haar measure. Finally, let $J_* : V \to (0,\infty)$ be the density of the push-forward via $\phi_*$ of the Haar measure with respect to the Lebesgue measure (so $J_* = J$ on $\Omega$), and define an isometric isomorphism $\Psi : L^2(G) \to L^2(V)$ by
\[
\Psi(F) = (F \circ \phi_*^{-1}) \, J_*^{1/2} .
\]
Since $A_\lambda := \Delta^\alpha L_{(f^{\lambda,1})^* * f^{\lambda,1}} \Delta^\alpha$ is a self-adjoint operator on $L^2(G)$, we can define a self-adjoint operator $\tilde A_\lambda$ on $L^2(\Lie{g}) = L^2(V) \oplus L^2(\Lie{g} \setminus V)$ by
\[
\tilde A_\lambda = \begin{pmatrix} \Psi A_\lambda \Psi^{-1} & 0 \\ 0 & 0 \end{pmatrix}
\]
and another self-adjoint operator $\hat A_\lambda$ on $L^2(\Lie{g})$ by $\hat A_\lambda = T_\lambda^{-1} \tilde A_\lambda T_\lambda$, where $T_\lambda$ is the isometry on $L^2(\Lie{g})$ defined by
\[
T_\lambda f(X) = \lambda^{-n/2} f(X/\lambda).
\]
It is now not difficult to check that, for all $\Phi \in C_{\pg}([0,\infty))$ and $g,h \in C_c(\Lie{g})$,
\begin{equation}\label{eq:nu_conj_fcinner}
\langle \Phi(\hat A_\lambda) g, h \rangle_{L^2(\Lie{g})} = \langle \Phi(A_\lambda) g^{\lambda,2}, h^{\lambda,2} \rangle_{L^2(G)} 
\end{equation}
for all $\lambda$ sufficiently small that $\supp T_\lambda g \cup \supp T_\lambda h \subseteq \Omega$.

For all $N \in \N$, from the cases $\Phi(u) = u^N$ and $\Phi(u) = u^{2N}$ of \eqref{eq:nu_continuousfc} and \eqref{eq:nu_conj_fcinner} it follows that, for all $g,h \in C_c(\Lie{g})$,
\begin{equation}\label{eq:nu_conj_conv}
\langle \hat A_\lambda^N g, h \rangle_{L^2(\Lie{g})} \to \langle A^N g,h \rangle_{L^2(\Lie{g})}, \qquad \| \hat A_\lambda^N g \|_{L^2(\Lie{g})} \to \| A^N g \|_{L^2(\Lie{g})}
\end{equation}
as $\lambda \to 0$, where $A := L_{f^* * f}$. In particular, from this and the density of $C_c(\Lie{g})$ in $L^2(\Lie{g})$ it is not difficult to conclude that, for all $g \in C_c(\Lie{g})$,
\begin{equation}\label{eq:nu_strongpower}
\hat A_\lambda^N g \to A^N g
\end{equation}
in $L^2$-norm as $\lambda \to 0$ \cite[Proposition 3.32]{brezis}. Since $A$ is a bounded self-adjoint operator on $L^2(\Lie{g})$, $C_c(\Lie{g})$ is a core for $A$ and \cite[Theorem 9.16]{weidmann} implies that 
\[
\hat A_\lambda \to A
\]
in the sense of strong resolvent convergence as $\lambda \to 0$. In turn this implies that, for all bounded continuous functions $\Phi : [0,\infty) \to \C$,
\begin{equation}\label{eq:nu_soc}
\Phi(\hat A_\lambda) \to \Phi(A)
\end{equation}
in the sense of strong operator convergence as $\lambda \to 0$ \cite[Theorem 9.17]{weidmann}.

Suppose now that $\Phi \in C_{\pg}([0,\infty))$. Then we can write $\Phi(u) = \tilde\Phi(u) \, (1+u^N)$ for some bounded continuous function $\tilde \Phi : [0,\infty) \to \C$ and $N \in \N$. For all $g,h \in C_c(\Lie{g})$, by \eqref{eq:nu_conj_fcinner},
\[
\langle \Phi(A_\lambda) g^{\lambda,2}, h^{\lambda,2} \rangle_{L^2(G)} = \langle \Phi(\hat A_\lambda) g, h \rangle_{L^2(\Lie{g})} = \langle \tilde\Phi(\hat A_\lambda) g, h \rangle_{L^2(\Lie{g})} + \langle \tilde\Phi(\hat A_\lambda) g, \hat A_\lambda^N h \rangle_{L^2(\Lie{g})}
\]
for all $\lambda$ sufficiently small, and the last quantity tends to
\[
\langle \tilde\Phi(A) g, h \rangle_{L^2(\Lie{g})} + \langle \tilde\Phi(A) g, A^N h \rangle_{L^2(\Lie{g})} = \langle \Phi(A) g, h \rangle_{L^2(\Lie{g})}
\]
as $\lambda \to 0$, by \eqref{eq:nu_strongpower} and \eqref{eq:nu_soc}.

\ref{en:nu_asspower}.
This is just a restatement of part \ref{en:nu_assfunccalc} in the case where $\Phi(u) = u^{q/2}$.
\end{proof}

We can finally prove the enunciated relation between Hausdorff--Young constants of a Lie group and its Lie algebra. We find it convenient to state the result together with its analogue for Young constants, since both follow by the approximation results of Lemma \ref{lem:nu_localisation}. Part \ref{en:hyloc} of Proposition \ref{prop:nu_loc}, together with the following Remark \ref{rmk:symmetry} and the Babenko--Beckner theorem for $\R^n$, prove Theorem \ref{thm:local-HY}.

As in \cite{Siebert-1986}, we define a \emph{contractive automorphism} of a locally compact group $G$ as an automorphism $\tau$ such that $\lim_{k\to\infty} \tau^k(x)=e$ for all $x \in G$.

\begin{proposition}\label{prop:nu_loc}
Let $G$ be a locally compact group.
\begin{enumerate}[label=(\roman*)]
\item\label{en:yloc} For all $p_1,\dots,p_k \in [1,\infty]$ such that $\sum_{j=1}^k 1/p_j' \in [0,1]$,
\begin{equation}\label{eq:ytrivialineq}
Y_{p_1,\dots,p_k}(G) \geq Y_{p_1,\dots,p_k}^\loc(G),
\end{equation}
with equality when $G$ has a contractive automorphism; moreover, if $G$ is a Lie group with Lie algebra $\Lie{g}$,
\begin{equation}\label{eq:ylocineq}
Y_{p_1,\dots,p_k}^\loc(G) \geq Y_{p_1,\dots,p_k}(\Lie{g}).
\end{equation}
\item\label{en:hyloc} 
For all $p \in [1,2]$,
\begin{equation}\label{eq:hytrivialineq}
H_p(G) \geq H_p^\loc(G),
\end{equation}
with equality if $G$ has a contractive automorphism.
Moreover, when $G$ is an $n$-dimensional Lie group with Lie algebra $\Lie{g}$,
\begin{equation}\label{eq:hylocineq}
H_p^\loc(G) \geq H_p(\Lie{g}).
\end{equation}
\end{enumerate}
\end{proposition}
\begin{proof}
\ref{en:yloc}.
The first inequality is obvious.
Moreover, in case $G$ has a contractive automorphism, the reverse inequality follows from a scaling argument. Indeed, for all automorphisms $\gamma$ of $G$, there exists $\kappa_\gamma \in (0,\infty)$ such that the push-forward via $\gamma$ of the Haar measure on $G$ is $\kappa_\gamma$ times the Haar measure. So, if $R_\gamma f = f \circ \gamma^{-1}$, then
\[
\| R_\gamma f \|_{L^p(G)} = \kappa_\gamma^{-1/p} \| f\|_{L^p(G)}, \quad R_\gamma \Delta = \Delta, \quad R_\gamma \left( \bigast_{j=1}^k f_j \right) = \kappa_\gamma^{k-1} \bigast_{j=1}^k R_\gamma f_j,
\]
whence it is immediate that both sides of Young's inequality \eqref{eq:nu_young} are scaled by the same factor when each $f_j$ is replaced with $R_\gamma f_j$. Now, by density, the value of the best constant $Y_{p_1,\dots,p_k}(G)$ may be determined by testing \eqref{eq:nu_young} on arbitrary $f_1,\dots,f_k \in C_c(G)$. Moreover, if $\tau$ is a contractive automorphism of $G$ and $U$ is any neighbourhood of the identity, then, for all compact subsets $K \subseteq G$, there exists $N \in \N$ such that $\tau^N(K) \subseteq U$ \cite[Lemma 1.4(iv)]{Siebert-1986}; in particular, for all $f_1,\dots,f_k \in C_c(G)$, by taking $\gamma = \tau^N$ for sufficiently large $N \in \N$, we see that $\supp R_\gamma f_j \subseteq U$. This shows that $Y_{p_1,\dots,p_k}(G) \leq Y_{p_1,\dots,p_k}(G;U)$ for all neighbourhoods $U$ of $e \in G$, and consequently $Y_{p_1,\dots,p_k}(G) \leq Y_{p_1,\dots,p_k}^\loc(G)$.

As for the second inequality, let $U$ be an arbitrary neighbourhood of $e \in G$.
To conclude, it is sufficient to show that $Y_{p_1,\dots,p_k}(\Lie{g}) \leq Y_{p_1, \dots, p_k}(G;U)$.

Let $r \in [1,\infty]$ be defined by $\sum_j 1/p_j' = 1/r'$. Consider $g,f_1,\dots,f_k \in C_c(\Lie{g})$.
For all $\lambda \in (0,\infty)$, $\alpha \in \C$ and $p\in[1,\infty]$, define $g^{\lambda,p},f_j^{\lambda,p},f_j^{\lambda,p,\alpha}$ as in Lemma \ref{lem:nu_localisation}.
Then
$\bigcup_{j=1}^k \supp f_j^{\lambda,1} \subseteq U$
for all sufficiently small $\lambda$,
and therefore, by \eqref{eq:nu_young},
\begin{multline*}
\left\langle 
\bigast_{j=1}^k (f_j^{\lambda,1} \Delta^{\sum_{l=1}^{j-1} 1/p_l'})
, g^{\lambda,\infty} \right\rangle_{L^2(G)} \\
\leq Y_{p_1, \dots, p_k}(G;U) \, \| f_1^{\lambda,1} \|_{L^{p_1}(G)} \cdots \|f_k^{\lambda,1} \|_{L^{p_k}(G)} \|g^{\lambda,\infty} \|_{L^{r'}(G)}.
\end{multline*}
Note that $\sum_{j=1}^k 1/p_j + 1/r' = k$.
So the last inequality can be rewritten as
\begin{multline*}
\left\langle f_1^{\lambda,1,\alpha_1} * \dots * f_k^{\lambda,1,\alpha_k} , g^{\lambda,\infty} \right\rangle_{L^2(G)} \\
\leq Y_{p_1, \dots, p_k}(G;U) \, \| f_1^{\lambda,p_1} \|_{L^{p_1}(G)} \cdots \|f_k^{\lambda,p_k} \|_{L^{p_k}(G)} \|g^{\lambda,r'} \|_{L^{r'}(G)},
\end{multline*}
where $\alpha_j = -\sum_{l=1}^{j-1} 1/p_l'$. Hence, by Lemma \ref{lem:nu_localisation}, by taking the limit as $\lambda \to 0$, we obtain
\[
\langle f_1 * \dots * f_k , g \rangle_{L^2(\Lie{g})} \leq Y_{p_1, \dots, p_k}(G;U) \, \| f_1\|_{L^{p_1}(\Lie{g})} \cdots \|f_k \|_{L^{p_k}(\Lie{g})} \|g \|_{L^{r'}(\Lie{g})}.
\]
The arbitrariness of $f_1,\dots,f_k,g \in C_c(\Lie{g})$ implies that $Y_{p_1,\dots,p_k}(\Lie{g}) \leq Y_{p_1, \dots, p_k}(G;U)$.

\ref{en:hyloc}.
Much as in part \ref{en:yloc},
the first inequality is obvious, and equality follows from a rescaling argument when $G$ has a contractive automorphism, since
\[
\| R_\gamma f \|_{\Four L^q(G)} = \kappa_\gamma^{-1/q'} \| f \|_{\Four L^q(G)}
\]
for all automorphisms $\gamma$ of $G$.

As for the second inequality, we need to show that $H_p(\Lie{g}) \leq H_p(G;U)$ for all neighbourhoods $U$ of $e \in G$.
Set $q = p'$ and note that, by \eqref{eq:nu_fouriernorm},
\[
\|f\|_{\Four F^{q}(G)}^{q}
= \sup_{\|g\|_{L^1(G)}, \|h\|_{L^1(G)} \leq 1} \langle |L_f \Delta^{1/q}|^{q} g, h \rangle_{L^2(G)}.
\]

For $\lambda \in (0,\infty)$, $r \in [1,\infty]$ and $f,g,h \in C_c(\Lie{g})$, we define $f^{\lambda,r},g^{\lambda,r},h^{\lambda,r} : G \to \C$ as in Lemma \ref{lem:nu_localisation}.
For all sufficiently small $\lambda$, $\supp f^{\lambda,r} \subseteq U$ and therefore
\[
\langle |L_{f^{\lambda,\infty}} \Delta^{1/q}|^{q} g^{\lambda,1}, h^{\lambda,1} \rangle_{L^2(G)} \leq H_p(G;U)^{q} \|f^{\lambda,\infty} \|_{L^p(G)}^{q} \|g^{\lambda,1}\|_{L^1(G)} \|h^{\lambda,1}\|_{L^1(G)},
\]
that is,
\begin{multline*}
\lambda^{-n(q-1)} \langle |L_{f^{\lambda,\infty}} \Delta^{1/q}|^{q} g^{\lambda,1}, h^{\lambda,1} \rangle_{L^2(G)}\\
 \leq H_p(G;U)^{q} \|f^{\lambda,p} \|_{L^p(G)}^{q} \|g^{\lambda,1}\|_{L^1(G)} \|h^{\lambda,1}\|_{L^1(G)}.
\end{multline*}
As $\lambda \to 0$, by Lemma \ref{lem:nu_localisation} we then deduce that
\[
\langle |L_{f}|^{q} g, h \rangle_{L^2(\Lie{g})} \leq H_p(G;U)^{q} \|f \|_{L^p(\Lie{g})}^{q} \|g\|_{L^1(\Lie{g})} \|h\|_{L^1(\Lie{g})}.
\]
By the arbitrariness of $g,h \in C_c(\Lie{g})$,
\[
\|f\|_{\Four L^q(\Lie{g})} \leq H_p(G;U) \|f\|_{L^p(\Lie{g})}
\]
and finally, by the arbitrariness of $f \in C_c(\Lie{g})$, $H_p(\Lie{g}) \leq H_p(G;U)$.
\end{proof}

\begin{remark}\label{rmk:symmetry}
The argument in Proposition \ref{prop:nu_loc} can be extended to the case of inequalities restricted to particular classes of functions on $G$.
In particular, suppose that the class of functions is determined by invariance with respect to the action of a compact group $K$ of automorphisms of $G$.
Then it is possible to choose a positive inner product on $\Lie{g}$ so that $K$ acts on $\Lie{g}$ by isometries (take any inner product on $\Lie{g}$ and average it with respect to the action of $K$), and the correspondence \eqref{eq:nu_associatedp} preserves $K$-invariance whenever $\Omega$ is a ball centred at the origin.
Moreover the class of functions on $\Lie{g}$ under consideration contains all radial functions.
Since the extremisers for Young and Hausdorff--Young constants on $\Lie{g}$ are centred gaussians \cite{Bec-1975,brascamp_best_1976}, which may be assumed to be radial, the resulting lower bounds do not change. This observation completes the proof of Theorem \ref{thm:local-HY}.
\end{remark}

\begin{remark}\label{rmk:Y_HY}
While the inequalities \eqref{eq:ytrivialineq} and \eqref{eq:hytrivialineq} may be strict for certain Lie groups $G$ (note that, when $G$ is compact, the global Young and Hausdorff--Young constants are equal to $1$), it appears natural to ask whether the inequalities \eqref{eq:ylocineq} and \eqref{eq:hylocineq} are actually equalities. We are not aware of any counterexample. As a matter of fact, a particular case of a recent result of Bennett, Bez, Buschenhenke, Cowling and Flock about nonlinear Brascamp--Lieb inequalities \cite{BBBCF_2018} entails that equality \emph{always} holds in \eqref{eq:ylocineq} for all Lie groups $G$:
\[
Y_{p_1,\dots,p_k}^\loc(G) = Y_{p_1,\dots,p_k}(\Lie{g})
\]
for \emph{all} $p_1,\dots,p_k \in [1,\infty]$ such that $\sum_{j=1}^k 1/p_j' \in [0,1]$.
By Proposition \ref{prop:basic}\ref{en:yheq}, this in turn implies that
\[
H_p^\loc(G) = H_p(\Lie{g}) = (B_p)^{\dim G}
\]
for all $p \in [1,2]$ such that $p'$ is an even integer, and \emph{a fortiori} the same equality holds for the $K$-invariant version of the constants for any compact group of automorphisms $K$.
\end{remark}

As a consequence of the above results, we strengthen some results of Klein and Russo \cite[Corollaries 2.5' and 2.8]{KR-1978}, where upper bounds for Young and Hausdorff--Young constants are obtained for particular solvable Lie groups.
Klein and Russo explicitly remark that they are able to obtain equalities instead of upper bounds in the particular case of the Heisenberg groups and only for special exponents (through a different argument, involving the analysis of the Weyl transform) and seem to leave the general case open.
Here instead we obtain equality for all the Young constants, as well as a lower bound for the Hausdorff--Young constants (which becomes an equality in the case of Babenko's exponents).

\begin{corollary}
Let $G$ be a $n$-dimensional solvable Lie group admitting a chain of closed subgroups
\[
\{e\} = G_0 < \dots < G_n = G,
\]
where $G_j$ is normal in $G_{j+1}$ and $G_{j+1}/G_j$ is isomorphic to $\R$.
Denote by $B_p$ the Babenko--Beckner constant.
Then the following hold.
\begin{enumerate}[label=(\roman*)]
\item\label{en:nu_ysolvable} for all $p_1,\dots,p_k,r \in [1,\infty]$ such that $\sum_{j=1}^k 1/p_j' = 1/r'$,
\[
Y_{p_1,\dots,p_k}(G) = Y^\loc_{p_1,\dots,p_k}(G) = (B_{r'} B_{p_1} \cdots B_{p_k})^n;
\]
\item\label{en:nu_hysolvable} for all $p \in [1,2]$,
\[
H_p(G) \geq H_p^\loc(G) \geq (B_p)^n,
\]
with equalities if $p' \in 2\Z$.
\end{enumerate}
\end{corollary}
\begin{proof}
\ref{en:nu_ysolvable}.
The inequality $Y_{p_1,\dots,p_k}(G) \leq (B_{r'} B_{p_1} \cdots B_{p_k})^n$ can be obtained, as in \cite{KR-1978}, by iteratively applying Proposition \ref{prop:basic}\ref{en:ext} and the fact that $Y_{p_1,\dots,p_k}(\R) = B_{r'} B_{p_1} \cdots B_{p_k}$ \cite{Bec-1975,brascamp_best_1976}.
On the other hand, by Propositions \ref{prop:nu_loc}\ref{en:yloc} and \ref{prop:basic}\ref{en:ext},
\[
Y_{p_1,\dots,p_k}(G) \geq Y^\loc_{p_1,\dots,p_k}(G) \geq Y_{p_1,\dots,p_k}(\Lie{g}) = Y_{p_1,\dots,p_k}(\R)^n = (B_{r'} B_{p_1} \cdots B_{p_k})^n,
\]
and we are done.

\ref{en:nu_hysolvable}.
From part \ref{en:nu_ysolvable} and Proposition \ref{prop:basic}\ref{en:yheq}, we deduce immediately  that $H_p(G) =  (B_p)^n$ whenever ${{q}}$ is an even integer.
On the other hand, by Proposition \ref{prop:nu_loc}\ref{en:hyloc},
\[
H_{p}(G) \geq H^\loc_{p}(G) \geq H_{p}(\Lie{g}) = (B_p)^n,
\]
by \cite{Bec-1975}, and we are done.
\end{proof}

\section{The \texorpdfstring{$n$}{n}-torus \texorpdfstring{$\T^n$}{Tn} revisited}\label{s:torus}

The proof of the central local Hausdorff--Young theorem on a compact Lie group mimics that of the local Hausdorff--Young theorem on $\T^n$, and we present this case first to make the proof of the general case more evident.

\begin{proof}[Proof of Theorem \ref{thm:local-HY-Tn}]
There is no loss of generality in supposing functions smooth; this ensures that all the sums and integrals that occur in the proof below converge.

Let us identify $\T^n$ with the subset $(-1/2, 1/2]^n$ of $\R^n$. For $f \in L^1(\T^n)$, the Fourier transform $\hat f : \Z^n \to \C$ of $f$ is given by
\[
\hat f(\mu) = \int_{\T^n} f(x) \, e^{2\pi i \mu \cdot x} \, dx.
\]
for all $\mu \in \Z^n$.
We denote by $V$ the open subset $(-1/2,1/2)^n$ of $\R^n$.
For any function $f \in L^1(\T^n)$ such that $\supp f \subseteq V$, we define $F$ on $\R^n$ by
\[
F(x) =
\begin{cases}
 f(x) & \text{when $x \in V$,} \\
 0     & \text{otherwise;}
\end{cases}
\]
we say that $F$ corresponds to $f$.
Clearly $F \in L^1(\R^n)$ and $\hat F|_{\Z^n} = \hat f$; further, if $f$ is smooth, so is $F$.
We are going to transfer the sharp Hausdorff--Young theorem for $F$ to $f$.

The Plancherel formulae for Fourier series and Fourier integrals imply that
\[
\| \hat f \|_{\ell^2(\Z^n)} = \| f \|_{L^2(\T^n)} = \| F \|_{L^2(\R^n)} = \| \hat F \|_{L^2(\R^n)}  .
\]
In particular, since $\hat F|_{\Z^n} = \hat f$,
\begin{equation}\label{eq:restriction}
\| \hat F|_{\Z^n}  \|_{\ell^2(\Z^n)} \leq \| \hat F \|_{L^2(\R^n)}  .
\end{equation}
Further, trivially,
\[
\|\hat F|_{\Z^n}  \|_{\ell^\infty(\Z^n)} \leq \| \hat F \|_{L^\infty(\R^n)}.
\]
If we could interpolate between these inequalities, then it would follow that
\begin{equation}\label{eq:key}
\| \hat F|_{\Z^n}  \|_{\ell^q(\Z^n)} \leq \| \hat F \|_{L^q(\R^n)}
\end{equation}
for all $q \in [2,\infty]$ and $\hat F$ in $L^{{{q}}}(\R^n)$, whence
\[
\| \hat f \|_{\ell^{p'}(\Z^n)} = \| \hat F|_{\Z^n} \|_{\ell^{p'}(\Z^n)} \leq \| \hat F \|_{L^{p'}(\R^n)} \leq (B_p)^n \| F \|_{L^p(\R^n)} = (B_p)^n \| f \|_{L^p(\T^n)} ,
\]
and we would be done.
But we can \emph{not} interpolate, because \eqref{eq:restriction} does not hold for all $\hat F$ in $L^{2}(\R^n)$, or even for all $\hat F$ in a dense subspace of $L^{2}(\R^n)$, but only for those $\hat F$ where $\supp F \subseteq V$; \emph{inter alia}, this ensures that $\hat F$ is smooth so that $\hat F|_{\Z^n}$ is well-defined.
So we prove a variant of \eqref{eq:key}.

Let $U$ be a small neighbourhood $U$ of $0$ in $\T^n$ such that $\overline U \subseteq V$, and take $\phi \in A(\R^n)$ such that $\supp \phi \subseteq V $ and $\phi(x) = 1$ for all $x \in U$.
We now define
\[
T G =  (\hat\phi * G)|_{\Z^n}
\qquad\forall G \in L^1(\R^n) + L^\infty(\R^n).
\]
We claim that when $q \in [2, \infty]$,
\begin{equation}\label{eq:key2}
\|  TG \|_{\ell^q(\Z^n)}  \leq  \| \hat\phi \|_{L^1(\R^n)} \| G \|_{L^q(\R^n)}
\qquad\forall G \in L^{q}(\R^n).
\end{equation}

To prove the claim, observe that the inverse Fourier transform of $\hat\phi * G$ is supported in $V$, whence
\begin{equation*}
\| TG \|_{\ell^2(\Z^n)}
 =    \| (\hat\phi * G)|_{\Z^n} \|_{\ell^2(\Z^n)}
\leq  \|  \hat\phi * G \|_{L^2(\R^n)}
\leq \| \hat\phi \|_{L^1(\R^n)} \| G \|_{L^2(\R^n)},
\end{equation*}
for all $G \in L^{2}(\R^n)$, by \eqref{eq:restriction}  and a standard convolution inequality.
Similarly, since $\hat\phi * G$ is continuous, the same inequalities hold when $2$ is replaced by $\infty$.
Thus \eqref{eq:key2} holds when $q$ is $2$ or $\infty$.
The Riesz--Thorin interpolation theorem establishes \eqref{eq:key2} for all $q \in [2, \infty]$.

To conclude the proof, take $f \in C^\infty(\T^n)$ such that $\supp f \subseteq U$, and let $F$ correspond to $f$.
Then $\hat F \in L^1(\R^n) \cap L^{\infty}(\R^n)$ and $\hat\phi * \hat F = \hat F$.
Thus
\[
\| \hat f \|_{\ell^q(\Z^n)} = \|  T\hat F  \|_{\ell^q(\Z^n)}  \leq  \| \hat\phi \|_{L^1(\R^n)} \| \hat F \|_{L^q(\R^n)}
\]
by \eqref{eq:key2}.
This now gives
\[\begin{split}
\|  \hat f \|_{\ell^{p'}(\Z^n)}
&\leq  \| \hat\phi \|_{L^1(\R^n)} \| \hat F \|_{L^{p'}(\R^n)}\\
&\leq  \| \hat\phi \|_{L^1(\R^n)} (B_p)^n \| F \|_{L^p(\R^n)}
=  \| \hat\phi \|_{L^1(\R^n)} (B_p)^n \| f \|_{L^p(\Z^n)}.
\end{split}\]
This proves that $H_p(\T^n;U) \leq \|\hat\phi\|_{L^1(\R^n)} (B_p)^n$.

By choosing $U$ small enough, we may make $\|\hat\phi \|_{L^1(\R^n)}$ as close to $1$ as we like (see \cite{leptin}): indeed, we can take $\phi = |K|^{-1} \chrfn_{U+K} * \chrfn_K$, where $K=-K$ is a fixed small neighbourhood of the origin (here $\chrfn_\Omega$ denotes the characteristic function of a measurable set $\Omega \subseteq \R^n$ and $|\Omega|$ its Lebesgue measure), so that $\supp \phi \subseteq U +2K$ and
\[
1 = \phi(0) \leq \|\hat\phi\|_{L^1(\R^n)} \leq |K|^{-1} \|\chrfn_K\|_{L^2(\R^n)} \|\chrfn_{U+K}\|_{L^2(\R^n)} = (|U+K|/|K|)^{1/2}.
\]
So $H_p^\loc(\T^n) \leq (B_p)^n$, and the converse inequality is given by Theorem \ref{thm:local-HY}.
\end{proof}

\section{Compact Lie groups}\label{s:compact}

Before entering into the proof of Theorem \ref{thm:local-central-HY-compact-Lie}, we present a summary of the theory of representations and characters of compact connected Lie groups $G$.
For more details, the reader may
consult, for example, \cite{BtD_1985,Knapp-2002}. We assume throughout that $G$ is not abelian, since the abelian case was treated in Theorem \ref{thm:local-HY-Tn}.

A compact connected Lie group $G$ comes with a set $\Lambda^+$ of 
\emph{dominant weights}, which 
parametrise the collection of irreducible unitary representations $\pi_\lambda$ of $G$ modulo equivalence.
Each such representation $\pi_\lambda$ is of finite dimension $d_\lambda$ and has a character $\chi_\lambda$ given by $\trace \pi_\lambda(\cdot)$.

Assume that the Haar measure on $G$ is normalised so as to have total mass $1$.
The Peter--Weyl theory gives us the
Plancherel formula: if $f \in L^2(G)$, then
\[
\| f\|_2^2 = \sum_{\lambda \in \Lambda^+} d_\lambda \| \pi_\lambda(f) \|_\HS^2  .
\]
In other words, the group Plancherel measure on the unitary dual of $G$ can be identified with the discrete measure on $\Lambda^+$ that assigns mass $d_\lambda$ to the point $\lambda$.
From the discussion in Section \ref{s:FLq}, we deduce that
\[
\| f \|_{\Four L^q} = \left( \sum_{\lambda \in \Lambda^+} d_\lambda \| \pi_\lambda(f) \|_{\Sch^q}^q \right)^{1/q}.
\]
for all $q \in [1,\infty)$. If $f$ is a central function, then $\pi_\lambda(f)$ is a multiple of the identity and
\[
\tilde f(\lambda) := \int_G f(x) \, \chi_\lambda(x) \,dx = \trace \pi_\lambda(f),
\]
whence
\[
\| f \|_{\Four L^q} = 
\left( \sum_{\lambda \in \Lambda^+} d_\lambda^{2-q} |\tilde f(\lambda)|^q \right)^{1/q}.
\]
For $q=2$, this corresponds to the fact that the characters $\chi_\lambda$ form an orthonormal basis for the space of square-integrable central functions.

A more precise description of the set $\Lambda^+$ of dominant weights and the characters $\chi_\lambda$ can be given as follows.
Recall that the conjugation action of the group $G$ on itself determines the adjoint representation of $G$ on $\Lie{g}$:
\[
\exp( \Ad(x) Y) = x \exp(Y) x^{-1}
\qquad\forall x \in G \quad\forall Y \in \Lie{g}.
\]
Since $G$ is compact, there exists an $\Ad(G)$-invariant inner product on $\Lie{g}$, which in turn determines a Lebesgue measure on $\Lie{g}$; we scale the inner product so that the Jacobian determinant $J : \Lie{g} \to \R$ of the exponential mapping is $1$ at the origin. Clearly $J$ is an $\Ad(G)$-invariant function.

The group $G$ contains a maximal torus $T$, that is, a maximal closed connected abelian subgroup, which is unique up to conjugacy; its Lie algebra $\Lie{t}$ is a maximal abelian Lie subalgebra of $\Lie{g}$.
The set $\Gamma$ of $X$ in $\Lie{t}$ such that $\exp X = e$ is a lattice in $\Lie{t}$, and $T$ may be identified with $\Lie{t} / \Gamma$. The \emph{weight lattice} $\Lambda$ is the dual lattice to $\Gamma$, that is, the set of elements $\lambda$ of the dual space $\Lie{t}^*$ taking integer values on $\Gamma$: equivalently, $\Lambda$ is the set of the $\lambda \in \Lie{t}^*$ such that $X \mapsto e^{2\pi i \lambda(X)}$ descends to a character $\kappa_\lambda$ of $T$. We say that a weight $\lambda \in \Lambda$ occurs in a unitary representation $\pi$ of $G$ if the character $\kappa_\lambda$ of $T$ is contained in the restriction of $\pi$ to $T$. Weights occurring in the (complexified) adjoint representation are called \emph{roots}.
A choice of an ordering splits roots into 
into \emph{positive} and \emph{negative} roots. We denote by $\rho$ half the sum of the positive roots. The set $\Lambda^+$ of \emph{dominant weights} is the set of the $\lambda \in \Lambda$ having nonnegative inner product with all positive roots. 
The irreducible representation $\pi_\lambda$ of $G$ corresponding to $\lambda \in \Lambda^+$ is determined, up to equivalence, by the fact that $\lambda$ is the highest weight occurring in $\pi_\lambda$ (that is, $\lambda$ occurs in $\pi_\lambda$, while $\lambda + \alpha$ does not occur in $\pi_\lambda$ for any positive root $\alpha$).

Via the orthogonal projection of $\Lie{g}$ onto $\Lie{t}$, we can identify $\Lie{t}^*$ with a subspace of $\Lie{g}^*$. Given $\lambda$ in $\Lie{g}^*$, we write $O_{\lambda}$ for the compact set $\Ad(G)^* \lambda$, usually called the \emph{orbit} of $\lambda$. Kirillov's character formula \cite[p.\ 459]{Kir} states that, for 
all $X \in \Lie{g}$ 
and all $\lambda \in \Lambda^+$,
\begin{equation}\label{eq:kirillov}
J(X)^{1/2} \, \chi_\lambda( \exp( X) ) =  \int_{O_{\lambda+\rho}} \exp( 2 \pi i \xi \cdot X) \,d\sigma(\xi),
\end{equation}
where $\sigma$ is a canonical $\Ad(G)^*$-invariant measure on $O_{\lambda+\rho}$, and $\xi \cdot X$ denotes the duality pairing between $\xi \in \Lie{g}^*$ and $X \in \Lie{g}$.
When $X = 0$, this formula becomes the normalisation
\[
\int_{O_{\lambda+\rho}} \,d\sigma(\xi) = d_\lambda .
\]

\begin{proof}[Proof of Theorem \ref{thm:local-central-HY-compact-Lie}]
Take a small connected conjugation-invariant neighbourhood $U$ of the identity in $G$ that is also symmetric, that is, $U^{-1} = U$.
Then $U = \bigcup_{x \in G} x (U \cap T) x^{-1}$.
Let $V$ be the small connected neighbourhood of $0$ in $\Lie{g}$ such that $U = \exp V$ and
$\exp$ is a diffeomorphism from a neighbourhood of $\overline V$ onto a neighbourhood of $\overline U$ in $G$.

To a function $f$ on $G$ supported in $U$, we associate the function $F$ on $\Lie{g}$ supported in $V$ by the formula
\[
F(X)  =
\begin{cases}
J(X)^{1/2} \, f(\exp(X))   &\text{when $X \in V$,} \\
0                     &\text{otherwise}.
\end{cases}
\]
Then 
$\| J^{1/p -1/2} F \|_p = \| f  \|_p$.
We define the Fourier transform $\hat F$ of $F$ as follows:
\[
\hat F (\xi ) = \int_{\Lie{g}} F(X) \, \exp( 2\pi i \xi \cdot X) \, dX \qquad  \forall \xi \in \Lie{g}^*.
\]
The following conditions are equivalent: $f$ is central on $G$; $F$ is $\Ad(G)$-invariant on $\Lie{g}$; and $\hat F$ is $\Ad(G)^*$-invariant on $\Lie{g}^*$.

Assume that $f$ is central and supported in $U$, and let $F$ be the associated function on $\Lie{g}$. From the character formula \eqref{eq:kirillov}, a change of variables, and a change of order of integration,
\begin{align*}
\tilde f(\lambda)
&= \int_G f(x) \, \chi_\lambda(x) \,dx
  = \int_{\Lie{g}}  F(X) \int_{O_{\lambda+\rho}} \exp( 2 \pi i \xi \cdot X) \,d\sigma(\xi)\,dX \\
&=  \int_{O_{\lambda+\rho}}  \int_{\Lie{g}} F(X) \exp(2 \pi i \xi \cdot X) \,dX\, d\sigma(\xi)
  =  \int_{O_{\lambda+\rho}}  \hat F(\xi) \, d\sigma(\xi)
 = d_\lambda \hat F(\lambda+\rho).
\end{align*}
This, combined with the Plancherel theorems for central functions on $G$ and for functions on $\Lie{g}$, implies that
\[
\sum_{\lambda \in \Lambda^+}   d_\lambda^2 |\hat F(\lambda+\rho)|^2
= \| f\|_2^2
= \| F \|_2^2 = \| \hat F \|_2^2.
\]
For such functions, moreover, $\hat F$ is continuous and so
\[
\sup_{\lambda \in \Lambda^+} | \hat F (\lambda + \rho) |_\infty \leq \| \hat F \|_\infty.
\]

For a function $H$ on $\Lie{g}^*$, we define
\[
H^G(\lambda) = \int_G H(\Ad(g)^*\lambda) \, dg.
\]
Much as in the case of $\T^n$, we choose an $\Ad(G)$-invariant function $\phi \in A(\Lie{g})$ which vanishes off $V$ and takes the value $1$ on the open $\Ad(G)$-invariant subset $W$ of $V$.
For $H$ in $L^1(\Lie{g}^*) + L^\infty(\Lie{g}^*)$, we define the function $TH$ by
\[
TH(\lambda) = \hat\phi*H^G(\lambda + \rho)
\qquad\forall\lambda \in \Lambda^+.
\]
For such functions $H$, the inverse Fourier transform $F$ of $\hat\phi*H^G$ is supported in $V$ and is $\Ad(G)$-invariant, so the corresponding function $f$ on $G$ is central and supported in $U$.
From our previous discussion,
\[
\left( \sum_{\lambda \in \Lambda^+}  d_\lambda^2 |  TH (\lambda)|^2 \right)^{1/2}
 = \| \hat\phi* H^G \|_{2}
 \leq \|\hat\phi\|_1 \| H^G \|_2 \leq \|\hat\phi\|_1 \| H \|_2
\]
and
\[
 \sup_{\lambda \in \Lambda^+}  | TH (\lambda)| \leq \| TH \|_\infty \leq \|\hat\phi\|_1 \| H^G \|_\infty \leq \|\hat\phi\|_1 \| H \|_\infty .
\]
By Riesz--Thorin interpolation, when $2 \leq q < \infty$,
\[
\left( \sum_{\lambda \in \Lambda^+}  d_\lambda^2 |  TH (\lambda)|^q \right)^{1/q}
 \leq \|\hat\phi\|_1 \| H \|_q .
\]

Much as in the proof of Theorem \ref{thm:local-HY-Tn}, if $f$ is a central function on $G$ supported in $\exp(W) \subseteq U$, and $F$ is the $\Ad(G)$-invariant function on $\Lie{g}$ corresponding to $f$, then
$T\hat F(\lambda) = \hat\phi * \hat F(\lambda+\rho) = \hat F(\lambda+\rho)$
 for all $\lambda \in \Lambda^+$. Hence, if $n=\dim G$, from the Hausdorff--Young inequality on $\R^n$ we deduce that
\begin{multline*}
\|f\|_{\Four L^{p'}}= \left( \sum_{\lambda \in \Lambda^+}  d_\lambda^{2-p'} |   \tilde f(\lambda) |^{p'} \right)^{1/{p'}}
  =  \left( \sum_{\lambda \in \Lambda^+}  d_\lambda^2 |   \hat F (\lambda+\rho) |^{p'} \right)^{1/{p'}} \\
\leq \| \hat\phi \|_1 \| \hat F \|_{p'} 
\leq \| \hat\phi \|_1 (B_p)^n  \| F \|_p
\leq \| \hat\phi \|_1 (B_p)^n \sup_{X \in W} J(X)^{1/2-1/p} \| f \|_p,
\end{multline*}
which shows that $H_{p,\Inn(G)}(G;\exp(W)) \leq \| \hat\phi \|_1 (B_p)^n \sup_{X \in W} J(X)^{1/2-1/p}$.
By taking $W$ small, we may make both $\sup_{X \in W} J(X)^{1/2-1/p}$ and $\| \hat\phi \|_1$ close to $1$. So $H_{p,\Inn(G)}^\loc(G) \leq (B_p)^n$, and the converse inequality is given by Theorem \ref{thm:local-HY}.
\end{proof}

\section{The Weyl transform}\label{s:weyl}

In this section, we shall  mostly adopt the notation from Folland's book \cite{Folland-1989}. 
The Weyl transform $\rho(f)$ of a function $f \in L^1(\C^n)$
can be written as the operator
\[
\rho(f) =\int_{\R^n} \int_{\R^n} f(u+iv) \, e^{2\pi i(uD+vX)}\, du \,dv
\]
on $L^2(\R^n)$,
where $uD=\sum_{j=1}^n u_j D_j$ and $vX=\sum_{j=1}^n v_j X_j$,  and where $D_j$ and $X_j$ denote the operators
\[
D_j\phi(x)=\frac 1{2\pi i}\frac \partial{\partial x_j} f(x)  \qquad\text{and}\qquad X_j\phi(x)=x_j\phi(x).
\]
Explicitly, $\rho(f)$ is the integral operator given by
\[
\rho(f) \phi(x)=\int_{\R^n} K_f(x,y) \, \phi(y) \, dy,
\]
with integral kernel given by
\[
K_f(x,y)=\int_{\R^n} f(y-x+iv) \, e^{\pi i v(x+y)} \, dv.
\]
As Folland points out on page 24 of his monograph, this notion of ``Weyl transform'' is historically incorrect---the Weyl transform of $f$ should rather be
$\rho(\hat f)$,  the pseudodifferential operator associated to  the symbol $f$ in the Weyl calculus \cite[Chapter 2]{Folland-1989}.
Nevertheless, we shall use the definition of Weyl transform above.

In \cite{KR-1978}, the authors consider the operator $\nu(f)$ given by
\[
\nu(f) =\int_{\R^n}\int_{\R^n} f(u +iv) \, e^{2\pi iuD} \, e^{2\pi ivX}\, du \, dv ,
\]
and call this the Weyl operator associated to $f$---this appears to be even more inappropriate, as $\nu(f)$ is actually more closely related to the Kohn--Nirenberg calculus (see, for example, \cite[(2.32)]{Folland-1989}). In any case, it is easily seen that the operators $\nu(f)$ and $\rho(f)$ are related by the identity
\begin{equation}\label{eq:nurho}
\nu(f)=\rho(e^{i\pi u \cdot v} f)
\end{equation}
 (compare also \cite[Proposition 2.33]{Folland-1989}).

We are interested in best constants in Hausdorff--Young inequalities of the form
\begin{equation}\label{eq:hyweyl}
\|\rho(f)\|_{\Sch^{p'}}\leq C \|f\|_{L^p(\C^n)},
\end{equation}
for suitable functions $f$,  for instance Schwartz functions.
In light of \eqref{eq:nurho}, we may work with $\nu(f)$ in place of $\rho(f)$ equally well. As discussed in the introduction, we denote by $W_p(\C^n)$ the best constant $C$ in \eqref{eq:hyweyl}, and use the symbols $W_p^\loc(\C^n)$, $W_{p,K}(\C^n)$ and $W_{p,K}^\loc(\C^n)$ for the corresponding local and $K$-invariant variants.

If $p=2$,  then  $\rho$ is indeed isometric from $L^2(\C^{n})$ onto the space of Hilbert--Schmidt operators \cite[Theorem (1.30)]{Folland-1989}, and thus the following ``Plancherel identity'' for the Weyl transform holds true:
\begin{equation}\label{eq:weyl_plancherel}
\|\rho(f)\|_{\HS}=\|f\|_2.
\end{equation}
This tells us that $W_2(\C^n) = 1$ and, by interpolation, $W_p(\C^n) \leq 1$ for all $p \in [1,2]$.

However, as Klein and Russo have shown, $W_p(\C^n) < 1$ when $1<p<2$.
Indeed, \cite[Theorem 1]{KR-1978} may be restated by saying that
\begin{equation}\label{eq:weylconstant}
W_p(\C^n) = (B_p)^{2n}
\end{equation}
when $p'\in 2\Z$.
Moreover, in contrast with the Euclidean case,
there are no extremal functions for the optimal estimate---the best constant can only be found as a limit, for instance along a suitable family of Gaussian functions $f$. This raises the question whether \eqref{eq:weylconstant} holds
for more general $p\in[1,2]$.

Besides being of interest in its own right, the determination of the best constants in the Hausdorff--Young inequality \eqref{eq:hyweyl} for the Weyl transform on $\C^n$ is relevant to the analysis of the analogous inequality on the Heisenberg group $\Heis_n$. Indeed, the proof of Klein and Russo \cite[Theorem 3]{KR-1978} that
\begin{equation}\label{eq:heisconstant}
H_p(\Heis_n) = (B_p)^{2n+1}
\end{equation}
when $p' \in 2\Z$ is based on a reduction, via a scaling argument, to the corresponding result \eqref{eq:weylconstant} for the Weyl transform. A somewhat refined version of the scaling argument, presented below, shows that the problem of determining the best Hausdorff--Young constants for the Heisenberg group is completely equivalent to the analogous problem for the Weyl transform, irrespective of the exponent $p \in [1,2]$, and also in case of restriction to functions with symmetries.

\begin{proposition}\label{prp:weylheisenberg}
For all compact subgroups $K$ of $\group{U}(n)$ and all $p \in [1,2]$,
\[
H_{p,K}(\Heis_n) = B_p W_{p,K}(\C^n).
\]
\end{proposition}
\begin{proof}
Let us identify $\Heis_n$ with $\C^n \times \R$ with group law
\[
(z,t) \cdot (z',t') = (z+z',t+t'+\Im(\bar z \cdot z')/2).
\]
The Lebesgue measure on $\C^n \times \R$ is a Haar measure on $\Heis_n$, which we fix throughout.
The Schr\"odinger representation $\pi$ of $\Heis_n$ on $L^2(\R^n)$ is given by
\[
\pi(u+iv,t) \phi(x) = e^{2\pi i t + 2\pi i v \cdot x + \pi i u \cdot v} \phi(u+x)
\]
\cite[(1.25)]{Folland-1989}.
For all $\lambda \in \R \setminus \{0\}$, the map $A_\lambda : \Heis_n \to \Heis_n$, given by
\[
A_\lambda(z,t) = \begin{cases}
(\sqrt{|\lambda|} \, z, \lambda t) &\text{if $\lambda > 0$,}\\
(\sqrt{|\lambda|} \, \bar z, \lambda t) &\text{if $\lambda < 0$},
\end{cases}
\]
is an automorphism of $\Heis_n$. The representations $\pi_\lambda = \pi \circ A_\lambda$
form a family of pairwise inequivalent irreducible unitary representations of $\Heis_n$, in terms of which we can express the Plancherel formula for $\Heis_n$:
\[
\| F \|^2_{L^2(\Heis_n)} = \int_{\R \setminus \{0\}} \| \pi_\lambda(F) \|_{\HS}^2 \, |\lambda|^n \,d\lambda. 
\]
\cite[p.\ 39]{Folland-1989}. Hence, by the discussion in Section \ref{s:FLq}, for all $q \in [1,\infty)$,
\begin{equation}\label{eq:FLq_Heis}
\| F \|^q_{\Four L^q} = \int_{\R \setminus \{0\}} \| \pi_\lambda(F) \|_{\Sch^q}^q \, |\lambda|^n \,d\lambda. 
\end{equation}

For all $F \in L^1(\Heis_n)$ and $\lambda \in \R$, let us set
\[
F^\lambda(z) = \int_\R F(z,t) \, e^{2\pi i t \lambda} \,dt.
\]
Then
\begin{equation}\label{eq:Heis_rep_weyl}
\pi_\lambda(F) = \rho(Z_\lambda F^\lambda),
\end{equation}
where, for all functions $f$ on $\C^n$,
\[
Z_\lambda f(z) = \begin{cases}
|\lambda|^{-n} f(|\lambda|^{-1/2} z) & \text{if $\lambda>0$,}\\
|\lambda|^{-n} f(|\lambda|^{-1/2} \bar z)  & \text{if $\lambda<0$.}
\end{cases}
\]

From the definition of $\rho$, it is not difficult to show that
\[
\rho(Z_{-1} f) = S \rho(f)^* S,
\]
where $S f(z) = f^*(z) = \overline{f(-z)}$. From this it readily follows that 
\begin{equation}\label{eq:weyl_conj}
\|\rho(Z_{-\lambda} f)\|_{\Sch^q} = \|\rho(Z_\lambda f)\|_{\Sch^q}
\end{equation}
for all $\lambda \in \R \setminus \{0\}$ and $q \in [1,\infty]$.

Let $F \in C^\infty_c(\Heis_n)$ be $K$-invariant. Then $Z_\lambda F^\lambda$ is also $K$-invariant for all $\lambda >0$. Hence, by \eqref{eq:FLq_Heis}, \eqref{eq:Heis_rep_weyl} and \eqref{eq:weyl_conj},
\[\begin{split}
\|F\|_{\Four L^{p'}} 
 &= \left( \int_{\R \setminus \{0\}} \| \rho(Z_{|\lambda|} F^\lambda) \|_{\Sch^{p'}}^{p'} \, |\lambda|^n \,d\lambda \right)^{1/p'} \\
&\leq W_{p,K}(\C^n) \left( \int_{\R \setminus \{0\}} \| Z_{|\lambda|} F^\lambda \|_{p}^{p'} \, |\lambda|^n \,d\lambda \right)^{1/p'} \\
&= W_{p,K}(\C^n) \left( \int_{\R \setminus \{0\}} \| F^\lambda \|_{p}^{p'}  \,d\lambda \right)^{1/p'} \\
&\leq W_{p,K}(\C^n) \left( \int_{\C^n} \left( \int_\R | F^\lambda(z) |^{p'}  \,d\lambda \right)^{p/p'} \,dz \right)^{1/p} \\
&\leq W_{p,K}(\C^n) B_p \|F\|_p,
\end{split}\]
where we applied, in order, the sharp Hausdorff--Young inequality for the Weyl transform and $K$-invariant functions, a scaling, the Minkowski integral inequality (note that $p'/p \geq 1$) and the sharp Hausdorff--Young inequality on $\R$. This shows that $H_{p,K}(\Heis_n) \leq B_p W_{p,K}(\C^n)$.

Conversely, let $f \in C^\infty_c(\C^n)$ be $K$-invariant and $\phi : \R \to \C$ be in the Schwartz class, and let $F = f \otimes \phi$. Then $F$ is also $K$-invariant, and moreover $F^\lambda = \hat\phi(\lambda) f$. So, by applying the sharp Hausdorff--Young inequality on $\Heis_n$ to $F$ we obtain that
\begin{equation}\label{eq:test_sharp_HY_heis_tensor}
\left( \int_{\R \setminus \{0\}} \| \rho(Z_\lambda f) \|_{\Sch^{p'}}^{p'} \, |\hat\phi(\lambda)|^{p'} |\lambda|^n \,d\lambda \right)^{1/p'} 
\leq H_{p,K}(\Heis_n) \|f\|_p \|\phi\|_p.
\end{equation}
For $\mu \in (0,\infty)$ and $\lambda_0 \in \R \setminus \{0\}$, take 
\[
\phi(t) = e^{-\pi \mu t^2 - 2\pi i t \lambda_0},
\]
so that
\[
\hat\phi(\lambda) = \mu^{-1/2} e^{-(\pi/\mu) (\lambda-\lambda_0)^2} \qquad\text{and}\qquad  \|\hat \phi\|_{p'} = B_p \|\phi\|_p,
\]
since gaussians are extremal functions for the Hausdorff--Young inequality on $\R$.
With this choice of $\phi$, the inequality \eqref{eq:test_sharp_HY_heis_tensor} can be rewritten as
\[
B_p ( R_f * \Phi_\mu (\lambda_0) )^{1/p'} \leq H_{p,K}(\Heis_n) \|f\|_p,
\]
where $*$ denotes convolution on $\R$ and
\[
R_f(\lambda) = \| \rho(Z_\lambda f) \|_{\Sch^{p'}}^{p'} |\lambda|^n, \qquad \Phi_\mu(\lambda) = \frac{e^{-(\pi p'/\mu) \lambda^2}}{\int_\R e^{-(\pi p'/\mu) s^2} \,ds}.
\]
Note that $\lambda \mapsto Z_\lambda f$ is continuous $\R \setminus \{0\} \to L^p(\C^n)$ and $\rho : L^p(\C^n) \to \Sch^{p'}(L^2(\R^n))$ is continuous too, 
 so $R_f$ is a continuous function on $\R \setminus \{0\}$. Moreover $\Phi_\mu$ is an approximate identity as $\mu \to 0$. Hence, by taking the limit as $\mu \to 0$,
we obtain
\[
B_p \sup_{\lambda \in\R \setminus \{0\}} \| \rho(Z_\lambda f) \|_{\Sch^{p'}}  |\lambda|^{n/p'} \leq  H_{p,K}(\Heis_n) \|f\|_p,
\]
which for $\lambda = 1$ gives
\[
B_p \| \rho(f) \|_{\Sch^{p'}}  \leq  H_{p,K}(\Heis_n) \|f\|_p,
\]
that is, $B_p W_{p,K}(\C^n) \leq H_{p,K}(\Heis_n)$.
\end{proof}

Let us come back to the question whether the identity \eqref{eq:weylconstant}  holds for arbitrary $p \in [1,2]$.
The following result, which allows for arbitrary $p$ but restricts the class of functions $f$ and, regrettably, also requires a weight in the $p$-norm, gives another indication that this might be true.
Recall that a function $f$ on $\C^{n}$ is \emph{polyradial} if
\[
f(z) = f_0(|z_1|,\dots,|z_n|),
\]
or, equivalently, if $f$ is invariant under the $n$-fold product group $\group{U}(1) \times \dots \times \group{U}(1)$.

\begin{proposition}\label{prop:locweyl} If $f\in C_c^\infty(\C^{n})$ is polyradial, then, for all $p \in [1,2]$,
\begin{equation}\label{eq:hyweyl2}
\|\rho(f)\|_{\Sch^{p'} }\leq (B_p)^{2n}\|f e^{(\pi/2) |\cdot|^2}\|_{L^p(\C^{n})}.
\end{equation}
\end{proposition}

As observed in the introduction, this inequality implies that $W^\loc_{p,K}(\C^n) \leq (B_p)^{2n}$ for $K = \group{U}(1) \times \dots \times \group{U}(1)$, and \emph{a fortiori} also for any larger group $K$.

\begin{proof}
We present a proof of Proposition \ref{prop:locweyl} which follows the philosophy of the proof of Theorem \ref{thm:local-central-HY-compact-Lie}. The key is the  following identity relating Laguerre polynomials to Bessel functions:
\begin{equation}\label{eq:bela1}
L^\alpha_k(x)=\frac {e^x x^{-\alpha/2}}{k!} \int_0^\infty t^{k+\alpha/2} \, J_\alpha(2\sqrt{xt}) \, e^{-t} \, dt \qquad\forall x>0,
\end{equation}
where $\alpha \in (-1,\infty)$ \cite[(4.19.3)]{lebedev}.
In order to avoid technicalities, let us concentrate on the case where $n=1$;  we shall later indicate the straightforward changes in the argument  which are needed to deal with  general $n\geq 1$.

If $f(z)=f_0(|z|)$ is a radial $L^1$-function on $\C$, then one may use the orthonormal basis of Hermite functions $h_k$ ($k\in \N$)  of $L^2(\R)$ to represent the operator $\rho(f)$ as an infinite diagonal
 matrix, with diagonal elements given by
\begin{equation}\label{eq:diag1}
\tilde f(k):=\langle \rho(f) h_k,h_k\rangle = \int_{\C} f(z) \, \chi_k(z) \, dz \qquad\forall k\in\N,
\end{equation}
where
$\chi_k$ is the Laguerre function
\[
\chi_k(z)=e^{-(\pi/2)|z|^2}L^0_k(\pi |z|^2).
\]
(see \cite[(1.45) and (1.104)]{Folland-1989}; see also \cite[(1.4.32)]{thangavelu_1998}). In particular,
\begin{equation}\label{eq:weyl_schatten_rad}
\|\rho(f)\|_{\Sch^q}=\|\tilde f\|_{\ell^q}
\end{equation}
for all $q \in [1,\infty]$.

Recall also  that the Euclidean Fourier transform of any radial $L^1$-function $g$ on $\C \cong \R^2$ can be written in polar coordinates as
\begin{equation}\label{eq:hatrad}
\hat g(\zeta)=2\pi \int_0^\infty g_0(r) \, J_0(2\pi |\zeta| r) r\, dr,
\end{equation}
where $g(z) = g_0(|z|)$.
We assume that $f$ has compact support, and put  $F(z)=e^{(\pi/2) |z|^2} f(z)$.
Since also $\hat F$ is radial, we may write $\hat F(\zeta)=\hat F_0(|\zeta|).$
Combining  \eqref{eq:bela1} and \eqref{eq:hatrad},   we obtain
\begin{equation}\label{eq:tk1}
\tilde f(k)=\int_0^\infty \hat F_0\big(\sqrt{t/\pi}\big)\,  \frac {t^k}{k!} e^{-t} \,dt,
\end{equation}
which can be re-written as
\begin{equation}\label{eq:tk2}
\tilde f(k)=\int_{\C} \hat F(\zeta)  \frac {\pi^k|\zeta|^{2k}}{k!} e^{-\pi|\zeta|^2} \,d\zeta=\int_{\C}\hat F(\zeta) \, d\mu_k(\zeta),
\end{equation}
 where the measures
$d\mu_k$, $k\in \N$,  are probability measures on $\C$.
Combining the aforementioned Plancherel identity for the Weyl transform, which leads to
\[
\sum_{k\in \N}\left|\int_{\C}\hat F(\zeta) \, d\mu_k(\zeta) \right|^2 = \|f\|_2^2 \leq \|F\|_2^2 =\|\hat F\|_2^2,
\]
with the trivial estimate
\[
\sup_{k\in\N} \left|\int_{\C}\hat F(\zeta) \, d\mu_k(\zeta) \right| \leq \|\hat F\|_\infty,
\]
we see that from here on we can easily modify the argument in the proof of Theorem \ref{thm:local-central-HY-compact-Lie} in order to arrive at \eqref{eq:hyweyl2}.

Indeed, an even simpler interpolation argument is possible here, which avoids any smallness assumption on the support of $f$.
For suitable functions $\phi$ on the positive real line, let us write
\[
\breve \phi(k)=\int_0^\infty \phi(t) \, \frac{t^k}{k!} \, e^{-t} \,dt
\]
for all $k \in \Z$. We claim that
\begin{equation}\label{eq:pqest}
\|\breve \phi\|_{\ell^q}\le \|\phi\|_{L^q(\R^+, dt)}
\end{equation}
for all $q \in [1,\infty]$. Indeed, this estimate is trivial for $q=\infty$,  since the $\frac{t^k}{k!} \, e^{-t} \,dt$ are probability measures, and for $q=1$ we may estimate as follows:
\[
\sum_{k=0}^\infty|\breve \phi(k)|\le\int_0^\infty |\phi(t)| \sum_{k=0}^\infty \frac{t^k}{k!} \, e^{-t} \,dt =\|\phi\|_1.
\]
Thus, \eqref{eq:pqest} follows by Riesz--Thorin interpolation.
From \eqref{eq:pqest} and \eqref{eq:tk1},
\begin{align*}
\|\tilde f \|_{\ell^q}\le \left(\int _0^\infty \left| \hat F_0\big(\sqrt{t/\pi}\big) \right|^q \, dt\right)^{1/q}=\|\hat F\|_q,
\end{align*}
and thus, by \eqref{eq:weyl_schatten_rad} and the sharp Hausdorff--Young inequality on $\R^2$, we obtain
\[
\|\rho(f)\|_{\Sch^{p'}}=\|\tilde f \|_{\ell^{p'}}\le (B_p)^2 \|F\|_p,
\]
whence \eqref{eq:hyweyl2} follows.

Let us finally indicate the changes needed to deal with the case of arbitrary $n$.
The Laguerre functions must be replaced by the $n$-fold tensor products
\[
\chi_k(z_1,\dots, z_n)=\chi_{k_1}(z_1) \dots \chi_{k_1}(z_1),\]
where $k=(k_1,\dots,k_n)\in \N^n$, and thus, in place of \eqref{eq:diag1},
\[
\tilde f(k)=\int_{\C^{n}} f(z_1,\dots, z_n) \,\chi_k(z_1,\dots, z_n)\, dz_1\dots dz_n
\]
where $k\in \N^n$.
Accordingly, the measures $d\mu_k$ must be replaced by the $n$-fold tensor products
$d\mu_k=d\mu_{k_1}\otimes \dots \otimes d\mu_{k_n}$, which are again probability measures, and so on.
It then becomes evident that the proof carries over without any difficulty to this general case.
\end{proof}

\begin{remark}\label{rem:newideasneeded}
There are indications that it may not be possible to establish \eqref{eq:hyweyl2} without the presence of the weight $e^{(\pi/2) |\cdot|^2}$ by means of a reduction to the Euclidean Fourier transform and the Babenko--Beckner estimate, and that new techniques are required.
Let us again restrict our discussion for simplicity to the case $n=1$.

There is another interesting identity relating Laguerre functions and Bessel functions, namely
\[
e^{-x/2} x^{\alpha/2}L^\alpha_k(x)=\frac{(-1)^k}2  \int_0^\infty  J_\alpha(\sqrt{xy}) \, e^{-y/2} \, y^{\alpha/2} \, L_k^\alpha(y) \,dy \qquad \forall x>0,
\]
where $\alpha \in (-1,\infty)$ \cite[(4.20.3)]{lebedev}.
For $\alpha=0$,  this in combination with \eqref{eq:hatrad} implies the well-known identity
\begin{equation}\label{eq:ch_ft}
\chi_k(z)=e^{-(\pi/2)|z|^2} L^0_k(\pi |z|^2)= \frac{(-1)^k}{2} \widehat{\chi_k}(z/2)
\end{equation}
(see \cite[Remark after Theorem (1.105)]{Folland-1989}, which is based on a more conceptual approach based on the Wigner transform).
This easily leads to  the identity
\begin{equation}\label{eq:tk3}
\tilde f(k)=\int_{\C} \hat f(\zeta) \,  (-1)^k \, 2 \, \chi_k(2\zeta)\, d\zeta=\int_{\C} \hat f(\zeta) \,d\nu_k(\zeta).
\end{equation}
In contrast with \eqref{eq:tk2}, the signed measure $d\nu_k$ oscillates when $k \geq 1$ and is no longer a probability measure.
Indeed, by \cite[Lemma 1]{markett}, we have
\[
\|\chi_k\|_1\sim k^{1/2}\quad \mbox{as}\quad  k\to \infty.
\]
Thus we cannot use \eqref{eq:tk3} in place of \eqref{eq:tk2} as before in order to get a sharp Hausdorff--Young estimate for $\rho(f)$ without a weight.

Even the case where $p'=2m$ for some $m \in \Z$ does not seem to allow one to reduce to the Euclidean estimate. Indeed, note that, for all $f \in L^1(\C^n)$,
\begin{equation}\label{eq:twisted_weyl}
\rho(f^*) = \rho(f)^* \qquad\text{and}\qquad \rho(f) \, \rho(g)=\rho(f\times g),
\end{equation}
where $f^*(z) = \overline{f(-z)}$ and $f\times g$ denotes the twisted convolution of $f$ and $g$, that is,
\begin{equation}\label{eq:twisted}
f \times g(z) = \int_{\C^n} f(z-w) \, g(w) \, e^{\pi i \Im (\bar z \cdot w)} \,dw
\end{equation}
\cite[(1.32)]{Folland-1989}.
In particular, if $f$ is radial and real-valued, then $f = f^*$ and therefore
\[
\|\rho(f)\|_{\Sch^{2m}}^{2m}=\|\rho(f)^m\|_{\HS}^2 =\| \rho(f\times \cdots \times f)\|_\HS^2 =\|f\times \cdots \times f\|_2^2,
\]
 with $m$ factors $f$.
A reduction to the sharp estimate for the Euclidean Fourier transform $\hat f$ of $f$ would therefore require the validity of an estimate of the form
\begin{equation}\label{nored}
\|f\times \cdots \times f\|_2\le \|\hat f\|_{2m}^m=\|f* \cdots * f\|_2,
\end{equation}
where $*$ denotes the Euclidean convolution.
However this estimate is false, even when $m=2$.

Indeed, it is sufficient to test the estimate \eqref{nored} when $f = \chi_k$.
Note that, from \eqref{eq:diag1} and the orthogonality of Laguerre polynomials,
\[
\tilde\chi_k(l) = \langle \rho(\chi_k) h_l, h_l \rangle = \langle \chi_k, \chi_l \rangle = \delta_{kl}.
\]
In particular $\rho(\chi_k \times \chi_k) = \rho(\chi_k)$, that is,
\[
\chi_k \times \chi_k = \chi_k.
\]
Therefore
\[
\|\chi_k \times \chi_k\|_2 = \|\chi_k\|_2 = 1,
\]
while
\[
\|\hat\chi_k\|_4 = 2^{-1/2} \|\chi_k\|_4 \sim k^{-1/4} (\log k)^{1/4} \quad\text{as } k \to \infty,
\]
by \eqref{eq:ch_ft} and \cite[Lemma 1]{markett}.
This shows that \eqref{nored} cannot hold when $m=2$ and for all radial real-valued functions $f$ (not even with some constant larger than one multiplying the right-hand side).
\end{remark}

In order to conclude the proof of Theorem \ref{thm:local-HY-weyl}, we need to prove the lower bound
\begin{equation}\label{eq:weyl_lowerbd}
W^\loc_K(\C^n) \geq (B_{p})^{2n}
\end{equation}
for any compact subgroup $K$ of $\group{U}(n)$. As we will see, this can be done much as in Section \ref{s:FLq}. For a function $f \in L^1(\C^n) + L^2(\C^n)$, let $T_f$ denote the operator of twisted convolution on the left by $f$, that is,
\[
T_f \phi = f \times \phi.
\]
In analogy with Proposition \ref{prp:fouriernorm}, we can characterise the Schatten norms of Weyl transforms $\rho(f)$ as follows.

\begin{proposition}
For all $q \in [2,\infty]$ and $f \in C_c(\C^n)$,
\[
\|\rho(f)\|_{\Sch^q}^q = \||T_f|^q\|_{L^1(\C^n) \to L^\infty(\C^n)}.
\]
\end{proposition}
\begin{proof}
From the Plancherel formula \eqref{eq:weyl_plancherel} for the Weyl transform, together with \eqref{eq:twisted_weyl},
it is easily seen that, for all $f \in L^1(G)$,
\[
\|\rho(f)\|_{L^2(\R^n)\to L^2(\R^n)} = \|T_f\|_{L^2(\C^n) \to L^2(\C^n)}.
\]
This corresponds to the well-known fact that the norm of a linear operator on $L^2(\R^n)$ is the same as the norm of the corresponding left-multiplication operator on $\HS(L^2(\R^n))$. Note, moreover, that the analogue of \eqref{eq:twisted_weyl} holds:
\[
T_{f^*} = T_f^* \qquad\text{and}\qquad T_{f \times g} = T_f T_g.
\]
Hence the correspondence $\rho(f) \mapsto T_f$ induces an isometric $*$-isomorphism between $\Lin(L^2(\R^n))$ and the von Neumann algebra of operators on $L^2(\C^n)$ generated by $\{ T_f \tc f \in L^1(\C^n) \}$.

Take now $f \in C_c(\C^n)$. Then $\rho(f) \in \Sch^q(\C^n)$ and
\[
\| \rho(f) \|_{\Sch^q(\C^n)}^q = \| |\rho(f)|^{q/2} \|_{\HS}^2.
\]
Since $|\rho(f)|^{q/2} \in \HS(L^2(\R^n))$, by the Plancherel theorem for the Weyl transform there exists $g \in L^2(\C^n)$ such that
\[
\rho(g) = |\rho(f)|^{q/2}.
\]
Since isomorphisms between von Neumann algebras preserve the polar decomposition and the functional calculus, 
\[
T_g = |T_f|^{q/2}.
\]
In order to conclude, then it is enough to show that
\[
\|\rho(g)\|_{\HS}^2 = \|T_g^2\|_{L^1(\C^n) \to L^\infty(\C^n)}.
\]
On the other hand, $T_g = |T_f|^{q/2}$ is a nonnegative self-adjoint operator, so
\[
\|T_g^2\|_{L^1(\C^n) \to L^\infty(\C^n)} = \|T_g\|^2_{L^1(\C^n) \to L^2(\C^n)}
\]
and, according to \eqref{eq:twisted}, $T_g$ is an integral operator with kernel $\tilde K_g$ given by
\[
\tilde K_g(z,w) = g(z-w) \, e^{\pi i \Im(\bar z \cdot w)},
\]
whence
\[
\|T_g\|_{L^1(\C^n) \to L^2(\C^n)} = \esssup_{w \in \C^n} \|\tilde K_g(\cdot,w)\|_2 = \|g\|_2 = \|\rho(g)\|_{\HS},
\]
and we are done.
\end{proof}

Given the above characterisation, the proof of the inequality \eqref{eq:weyl_lowerbd} proceeds, much as in Section \ref{s:FLq}, via a ``blow-up'' argument. The main observation here is that, if $S_\lambda$ denotes the $L^1$-isometric scaling on $\C^n$,
\[
S_\lambda f(z) = \lambda^{-2n} f(z/\lambda),
\]
then
\[
(S_\lambda f) \times (S_\lambda g) = S_\lambda(f \times_\lambda g),
\]
where
\[
f \times_\lambda g(z) = \int_{\C^n} f(z-w) \, g(w) \, e^{\pi i \lambda^2 \Im (\bar z \cdot w)} \,dw;
\]
moreover, from the above formula it is clear that, as $\lambda \to 0$, the scaled twisted convolution $\times_\lambda$ tends to the standard convolution on $\C^n \cong \R^{2n}$ (see also \cite{cowling_1981}). Following this idea, it is not difficult to prove the analogues of Lemma \ref{lem:nu_localisation} and Proposition \ref{prop:nu_loc}, where the twisted convolution $\times$ and the standard convolution on $\C^n$ take the place of the convolutions on the Lie group and the Lie algebra respectively. In addition, the action of $\group{U}(n)$ on functions on $\C^n$ commutes with the scaling operators $S_\lambda$ and the twisted convolution, so the analogue of Remark \ref{rmk:symmetry} applies here. We leave the details to the interested reader.

\begin{remark}
Given the noncommutative subject of this paper, it is natural to ask whether the best constants $H_p(G), H^\loc_p(G),\dots$ are the same in the category of operator spaces (that is, quantized or noncommutative Banach spaces). To be more precise, let us equip the (commutative and noncommutative) $L^q$-spaces involved in the corresponding Hausdorff--Young inequality with their natural operator space structures \cite{Pisier-2003}. Does the complete $L^p \to L^{p'}$ norm of the Fourier transform coincide with the corresponding norm $H_p(G)$ in the category of Banach spaces? In the Euclidean case of $H_p(\R^n)$, this problem was asked by Pisier in 2002 to the fourth-named author, but it is still open. \'Eric Ricard recently noticed that such a result for the Euclidean Fourier transform (that is, its completely bounded norm is still given by the Babenko--Beckner constant raised to the dimension of the underlying space) would give the expected constants for the Weyl transform in CCR algebras and, therefore, also for the Fourier transform in the Heisenberg group. Unfortunately, Beckner's original strategy crucially uses hypercontractivity, which has been recently proved to fail in the completely bounded setting \cite{AAA}. In conclusion, the above discussion indicates one more time (see Remark \ref{rem:newideasneeded}) that some new ideas seem to be necessary to solve these questions.
\end{remark}

\newcommand{\arttitle}[1]{\lq #1\rq,}
\newcommand{\journal}[1]{\textit{#1} }
\newcommand{\jvolume}[1]{\textbf{{#1}}}
\newcommand{\booktitle}[1]{\emph{{#1}.}}
\providecommand{\bysame}{\leavevmode\hbox to3em{\hrulefill}\thinspace}

\end{document}